\documentclass[12pt,reqno,a4paper]{amsart}
\usepackage{amssymb,amsfonts,latexsym,mathrsfs, tikz}

\usepackage{hyperref}
\usepackage[nameinlink]{cleveref}
\usepackage{mathtools}

\title[The dual of the Tits cone and dominance]{The dual of the Tits cone and dominance of Coxeter groups}

\author[X. Fu]{Xiang Fu}
\address{(Xiang Fu)\newline \indent Beijing International Center for Mathematical Research, Peking University, Beijing, China}
\email{fuxiang@math.pku.edu.cn}
\author[J. Gu]{Jiaqing Gu}
\address{(Jiaqing Gu)\newline \indent Beijing International Center for Mathematical Research, Peking University, Beijing, China}
\email{2001110007@stu.pku.edu.cn}
\author[M. Lyu]{Mengfan Lyu}
\address{(Mengfan Lyu) \newline \indent School of Computer, Data and Mathematical Sciences, Western Sydney University, Parramatta, Australia}
\email{Mengfan.Lyu@westernsydney.edu.au}
\author[L. Reeves]{Lawrence Reeves}
\address{(Lawrence Reeves) \newline \indent School of Mathematics and Statistics, University of Melbourne, Melbourne, Australia}
\email{lreeves@unimelb.edu.au}
\author[L. Xu]{Linxiao Xu}
\address{(Linxiao Xu) \newline \indent Department of Foundational Mathematics, School of Mathematics and Physics,  Xi'an Jiaotong-Liverpool University, Suzhou, China}
\email{Linxiao.Xu@xjtlu.edu.cn}
\author[S. Zhao]{Shuyang Zhao}
\address{(Shuyang Zhao) \newline \indent School of Mathematical Sciences, Peking University, Beijing, China}
\email{zsy0509@pku.edu.cn}

\newtheorem{theorem}{Theorem}[section]
\newtheorem{lemma}[theorem]{Lemma}
\newtheorem{proposition}[theorem]{Proposition}
\newtheorem{corollary}[theorem]{Corollary}

\theoremstyle{definition}
\newtheorem{definition}[theorem]{Definition}

\theoremstyle{remark}
\newtheorem{remark}[theorem]{Remark}

\numberwithin{equation}{section}

\newcommand{\Z}{\mathbb{Z}}
\newcommand{\N}{\mathbb{N}}
\newcommand{\R}{\mathbb{R}}

\DeclareMathOperator{\dom}{\, dom}
\DeclareMathOperator{\PLC}{PLC}
\DeclareMathOperator{\coeff}{coeff}
\DeclareMathOperator{\pos}{Pos}

\DeclareMathOperator{\supp}{supp}
\DeclareMathOperator{\GL}{GL}

\DeclareMathOperator{\dep}{dp}
\DeclareMathOperator{\rad}{Rad}

\DeclareMathOperator{\conv}{conv}

\DeclareMathOperator{\Ext}{E_{ext}}


\begin{document}

\begin{abstract}
In this paper we introduce a potentially infinite process of reflections applied to a sequence of points in the dual of the Tits cone associated with an arbitrary finitely generated infinite Coxeter group. This process is infinite whenever the starting point is not in the imaginary cone of the given Coxeter group. We show that this process exhibits well-defined asymptotic behaviour. We study the resulting limits, and give a characterisation of these limits.  We then apply the techniques developed from this asymptotic process to the study of infinite sequences of roots totally ordered by the dominance partial ordering.      
\end{abstract}

\maketitle

\section{Introduction}

Let $(W,R)$ be a Coxeter system of finite rank (that is, $|R|<\infty$), with a root
system $\Phi$ arising from the Tits representation on a Coxeter datum $(V, \Pi, (\,\cdot ,\cdot\,))$  
in the sense of \cite{FU1} and \cite{FU2}. 
 When $W$ is infinite, the root system $\Phi$ is infinite, and its
large-scale geometry may be studied through the accumulation
directions of normalised roots (these are the projections of roots onto suitably chosen hyperplanes). 
These directions form the set of limit roots. Together with the imaginary cone and the dual of the
Tits cone, they provide a framework for describing
the asymptotic behaviour of normalised roots.

Another way of organising the roots is provided by a partial ordering called the  \emph{dominance}
order. Recall that a root $x$ dominates a root $y$ if every element of
$W$ that sends $x$ to a negative root also sends $y$ to a negative
root. Thus dominance is defined in terms of the inversion behaviour
of roots, whereas limit roots and the imaginary cone arise from their
asymptotic geometry. The purpose of this paper is to investigate a
closer connection between these two points of view. More precisely,
we study accumulation points of infinite dominance chains through the
geometry and dynamics of the dual Tits cone. For background on these
objects, see, for example,
\cite{BH93,BB98,Dyer12,DHR13,FU1,FU2,HLR11}.

We use the following realisation of the dual of the Tits cone:
$$
 U^*=\bigcap_{w\in W}
 w\bigl(\PLC(\Pi)\cup\{0\}\bigr).
$$
It is a closed, convex, $W$-invariant cone contained in the positive
cone $\PLC(\Pi)\cup \{0\}$. For $v\in V$, let
$$
 \pos(v):=\{x\in\Phi^+\mid(v,x)>0\}.
$$
For $\eta\in U^*$, the set $\pos(\eta)$ measures the failure of $\eta$ to lie in a chamber in which
all positive roots pair non-positively with it. Its cardinality also
distinguishes elements of the imaginary cone inside $U^*$.

We begin by establishing two basic geometric properties of $U^*$.
First, if $\eta,\eta'\in U^*$, then
$$
 (\eta,\eta')\leq0.
$$
Second, for every $\eta\in U^*$ and every $\varepsilon>0$, the set
$$
 \{x\in\Phi^+\mid(\eta,x)\geq\varepsilon\}
$$
is finite. Thus an infinite sequence of distinct positive roots
cannot pair uniformly positively with a fixed element of $U^*$.
These properties provide the finiteness needed to control sequences
of reflections and their limiting behaviour.

Our first main construction is a reduction procedure on $U^*$.
Starting with $\eta_1:=\eta\in U^*$, choose, whenever possible, a
simple root $a_i$ with
$$
 a_i\in\Pi\cap\pos(\eta_i)
$$
and set
$$
 \eta_{i+1}:=r_{a_i}\eta_i.
$$
We call the resulting sequence $(a_i)_{1\leq i\leq |\pos(\eta)|}$ a \emph{transitional sequence} and the corresponding sequence 
$(\eta_i)_{1\leq i\leq |\pos(\eta)|+1}$ a \emph{descending sequence}. When
$\pos(\eta)$ is finite, this procedure terminates after finitely many
steps. For a general element of $U^*$, however, the procedure may
continue indefinitely and should instead be regarded as an
asymptotic reduction.

Since
$r_{a_i}\eta_i=\eta_i-2(\eta_i,a_i)a_i$,
the coefficient of each simple root is non-increasing along a descending
sequence (hence the terminology). All coefficients remain non-negative because every
$\eta_i$ belongs to $U^*$. It follows that every infinite
descending sequence has a well-defined limit
$$
 \eta_{\mathbf a}:=\lim_{i\to\infty}\eta_i.
$$
We describe the support and positive set of this limit in terms of
the simple roots that occur infinitely often in the sequence. In
particular, the supports of the $\eta_i$ eventually stabilise, and
the support of $\eta_{\mathbf a}$ is a union of connected components
of the eventual support. For isotropic elements of $U^*$, the
construction is more rigid: under an appropriate connectedness
hypothesis, the limit is zero, and every simple root in the eventual
support occurs infinitely often.

It should be pointed out that the limit $\eta_{\mathbf{a}}$ depends on the specific 
transitional sequence $\mathbf{a}=(a_i)_{1\leq i\leq |\pos(\eta)|}$.  
In this paper we completely characterise when different transitional sequences may lead to the same limit 
starting from a point $\eta\in U^*$.

The rest of this paper is devoted to an application of the construction of infinite transitional
sequences and descending sequences to infinite injective sequences of positive roots totally ordered by
dominance.  
Some of the properties of the dominance ordering have been key components in establishing the (bi)automaticity for finitely generated Coxeter groups. 
In this paper, we apply the techniques developed in the study of the dual of the Tits cone to investigate infinite strict dominance chains
$$
 x_{i+1}\dom x_i,\qquad i\geq1.
$$
After normalisation, the roots $\widehat{x_i}$ lie in the compact
simplex $\conv(\Pi)$ (whenever $|\Pi|<\infty$, that is, whenever $W$ is finitely generated). 
Consequently, such a chain has accumulation
points among the limit roots. It is known that these accumulation points
also lie in the normalised dual Tits cone, see, for instance, \cite{DHR13}. 
In \cite[Section 7.7]{DHR13} the following question was asked: whether or not a normalised sequence corresponding to an infinite injective sequence of roots totally ordered by dominance (an infinite dominance chain) possesses a unique point of accumulation. In a forthcoming paper, 
we resolve this question in full. Still, in our view, it is worthwhile to see how the asymptotic action on the Tits cone may be applied to the study of infinite dominance chains. In the present paper, we present a number of useful observations on the descending chains starting from potential accumulation points of infinite dominance chains.
  
In particular, we show that the supports of the roots in a dominance
chain eventually stabilise and coincide with the support of every
accumulation point. It follows that all accumulation points of a
given chain have the same connected support. We also derive
inequalities for the bilinear form along dominance chains and use
them to compare the positive sets of their accumulation points.

One of the principal results concerns two suitably interlaced
dominance chains. We prove that accumulation points $\eta_1$ and
$\eta_2$ arising from such chains determine the same bilinear form $(\cdot, \cdot)$ pattern on
the root system:
$$
 (\eta_1,x)\diamond 0
 \quad\Longleftrightarrow\quad
 (\eta_2,x)\diamond0
 \qquad\text{for every }x\in \Phi,
$$
where $\diamond$ can assume each of the $<$, $=$ and $>$ relations, respectively.
Moreover these patterns generalise to the dual
Tits cone:
$$
 (\eta_1,\eta')\diamond 0
 \quad\Longleftrightarrow\quad
 (\eta_2,\eta')\diamond0
 \qquad\text{for every }\eta'\in U^*,
$$
where $\diamond$ can assume each of the $<$ and $=$  relations, respectively (note that 
$(u, v)\leq 0$, for all $u, v\in U^*$).

Finally, we study when unique accumulation must take place. Let
$$
 \overline Z=\conv(E)
$$
be the imaginary convex body. If a normalised dominance chain
converges to $\eta\in E$, we relate the asymptotic behaviour of
$$
 (\eta',x_i)|x_i|_1,
 \qquad \eta'\in U^*,
$$
to the convex geometry of $\overline Z$ at $\eta$. This yields
criteria involving extreme points of $\overline Z$ and, under
additional positivity hypotheses, sufficient conditions for two
interlaced dominance chains to have the same limit.

The paper is organised as follows. In the next section we recall the necessary background on based root systems, the Tits cone, the imaginary cone, limit roots, and dominance. We then establish some elementary properties of $U^*$. Section~\ref{sec:asymptotic-action}
introduces transitional sequences and descending sequences and studies their limits. Section~\ref{sec:dominance-accumulation} applies these results to accumulation points of infinite dominance chains. The final part investigates extreme points of the imaginary convex body and gives criteria for the uniqueness of limits.

\section{Preliminaries and notation}

\subsection{Standing assumptions and notation}

Throughout this paper, $(W,R)$ is a Coxeter system of finite rank. We work with a based root system $(\Phi,\Pi)$ in a finite-dimensional real vector space $V$, where
$$
\Pi=\{a_s\mid s\in R\}
$$
is assumed to be a basis of $V$ whose elements are in bijective correspondence with elements of $R$. The space $V$ is equipped with a symmetric bilinear form $(\,\cdot\,,\,\cdot\,)$ satisfying
$$
(a_s,a_s)=1
$$
for every $s\in R$, and, for distinct $s,t\in R$,
$$
(a_s, a_t)=
\begin{cases}
-\cos(\pi/m_{st}),&m_{st}<\infty,\\
\leq -1,&m_{st}=\infty.
\end{cases}
$$
Here $m_{st}$ denotes the order of $st$ in $W$. The reflection $r_{a_s}\in \GL(V)$ corresponding to $a_s\in\Pi$ acts on $V$ by
$$
r_{a_s}(v)=v-2(v,a_s)a_s.
$$
Standard Tits representation theory establishes that
$$W \to \langle r_a\mid a\in \Pi\rangle\subset \GL(V)$$ 
where $r\mapsto r_{a_r}$ for all $r\in R$ is an isomorphism (where $\langle r_a\mid a\in \Pi\rangle$ denotes the subgroup of $\GL(V)$ generated by $r_a$, for all $a\in \Pi$). 
We identify $R$ with $\{r_a\mid a\in\Pi\}$ and then $W$ acts on $V$ via $\GL(V)$. Moreover, we write
$$
\Phi=W\Pi
$$
for the associated root system. We refer to \cite{BN68, ABFB, HM} for the standard background on Coxeter groups, based root systems and \cite{Dyer12,DHR13} for background on imaginary cones, and limit roots.

For a subset $A\subseteq\Pi$, set
$$
\PLC(A):=
\left\{
 \sum_{a\in A}\lambda_a a
 \,\middle|\,
 \lambda_a\geq0\text{ $\forall a\in A$ }
 \text{ and }\lambda_a>0\text{ for at least one }a
\right\}.
$$
The sets of positive and negative roots are
$$
\Phi^+:=\Phi\cap\PLC(\Pi),
\qquad
\Phi^-:=-\Phi^+.
$$
Thus $\Phi=\Phi^+\uplus\Phi^-$ (where $\uplus$ denotes disjoint union). If
$$
v=\sum_{a\in\Pi}\lambda_a a\in V,\quad \lambda_a\in \R,\,\forall a\in \Pi,
$$
we write
$$
\coeff_a(v):=\lambda_a,
\qquad
\supp(v):=\{a\in\Pi\mid \lambda_a\neq0\},
\qquad
|v|_1:=\sum_{a\in\Pi}\lambda_a.
$$
Notice that $|\cdot|_1$ is a linear functional on $V$; it agrees with the usual $\ell^1$-norm on $\PLC(\Pi)\cup\{0\}$. In particular, $|v|_1>0$ for every $v\in\PLC(\Pi)$.

We recall the following useful result (\cite[Corollary 2.11]{HLR11}):
\begin{lemma}
\label{normbd}
When $(W, R)$ is a Coxeter system with $|R|<\infty$, for each finite $M>0$ the following set
$$\big\{x\in \Phi^+\colon |x|_1\leq M \big\}$$
is finite.
\qed
\end{lemma}

The Coxeter graph will be regarded as a graph with vertex set $\Pi$, in which distinct vertices $a,b\in\Pi$ are adjacent when $(a,b)\neq0$.
Accordingly, a subset of $\Pi$ is called connected if the corresponding induced subgraph is connected.

For $v\in V$, define
$$
\pos(v):=\{x\in\Phi^+\mid (v,x)>0\}.
$$
We also write
$$
Q:=\{v\in V\mid (v,v)=0\}
$$
for the isotropic cone.

For $w\in W$, its inversion set is
$$
N(w):=\{x\in\Phi^+\mid wx\in\Phi^-\}.
$$
The depth of a positive root $x$ is
$$
\dep(x):=\min\{\ell(w)\mid w\in W,\;wx\in\Phi^-\}.
$$
We recall the following useful result on depth (\cite[Lemma 1.7]{BH93}):
\begin{lemma}
\label{depthchange}
For $x\in \Phi^+$ and $a\in \Pi$ with $a\neq x$ we have 
\begin{equation*}
\dep(r_a x) =
\begin{cases} \dep(x)-1 & \text{if } (x, a) > 0,\\
\dep(x)  & \text{if } (x, a) = 0,\\
\dep(x)+1 & \text{if } (x, a) < 0.
\end{cases}
\end{equation*}
\qed

\end{lemma}

For roots $x,y\in\Phi$, we say that $x$ \emph{dominates} $y$, and write $x\dom y$, if
$$
wx\in\Phi^-\quad\Longrightarrow\quad wy\in\Phi^-
\qquad\text{for every }w\in W.
$$
The dominance is called strict if, in addition, $x\neq y$. For
$x\in\Phi^+$, we write
$$
D(x):=\{y\in\Phi^+\mid x\dom y\}.
$$

We call a positive root elementary if it strictly dominates no positive
root, and write
$$
D_0:=\{x\in\Phi^+\mid D(x)=\{x\}\}.
$$
In particular, $\Pi\subseteq D_0$.

We recall the following useful and well-known results on dominance (\cite[Lemma 2.3]{BH93} and \cite[Lemma 3.2]{FU1}):

\begin{lemma}
\label{domandform}
Suppose that $x, y\in \Phi$. Then there is dominance between $x$ and $y$ if and only if
$$(x, y)\geq 1.$$
\qed
\end{lemma}
Next, we have an easy observation on dominance which may be helpful later on.
\begin{lemma}
\label{dfin}
For each $x\in \Phi^+$, 
$$D(x)\subseteq N(w),$$
whenever $wx\in \Phi^-$, and, in particular, $|D(x)|<\infty$.
\end{lemma}
\begin{proof}
By the definition of dominance, a positive root $y\in D(x)$ if and only if $wy\in \Phi^-$ whenever $wx\in \Phi^-$,
whence $D(x)\subseteq N(w)$ whenever $wx\in \Phi^-$.
\end{proof}

\subsection{Affine and non-affine dihedral reflection subgroups}
It is well known that $W':=\langle r_x, r_y\rangle$ (where $x, y\in \Phi$ and $x\neq \pm y$), the \emph{reflection subgroup} generated by a pair of distinct reflections $r_x$ and $r_y$, is isomorphic to the dihedral Coxeter group. It is also well known that 
$$|W'|<\infty \iff -1<(x, y)<1.$$
While when $|W'|=\infty$, we distinguish two types of infinite dihedral reflection subgroups, namely, the affine dihedral reflection subgroups and infinite non-affine dihedral reflection subgroups. Specifically, 
$$\text{$W'$ is affine dihedral iff $(x, y)=\pm 1$};$$
whereas 
$$\text{$W'$ is infinite non-affine dihedral iff $|(x, y)|>1$}.$$
Both affine dihedral reflection subgroup and infinite non-affine dihedral reflection subgroup  are isomorphic to the abstract affine dihedral Coxeter group. 
For general references on reflection subgroups of Coxeter groups (these are subgroups of Coxeter groups which are generated by conjugates of the Coxeter generators) we refer the readers to consult classical literatures, including \cite{VD82}, \cite{VD89}, \cite{MD87} and \cite{MD90}. For how to locate affine reflection subgroup within a Coxeter group, we refer the readers to consult \cite{FRX}. 

\subsection{Normalisation and limit roots}

Let
$$
V_1:=\{v\in V\colon |v|_1=1\}.
$$
Whenever $|v|_1\neq0$, define the normalisation of $v$ by
$$
\widehat v:=\frac{v}{|v|_1}\in V_1.
$$
In particular,
$$
\widehat{\Phi}:=\{\widehat x\mid x\in\Phi\}
               =\{\widehat x\mid x\in\Phi^+\}
               \subseteq\conv(\Pi).
$$
The set of \emph{limit roots}, denoted by $E$, is the set of accumulation points of $\widehat{\Phi}$ in $V_1$. Since $\Pi$ is finite, $\conv(\Pi)$ is compact (being closed and bounded in a finite dimensional vector space over the real numbers), and hence every infinite injective sequence of positive roots has a normalised subsequence converging to an element of $E$.
Moreover,
$$
E\subseteq Q\cap V_1.
$$

For $w\in W$ and $v\in V_1$ such that $|wv|_1\neq0$, define the
normalised action by
$$
w\cdot v:=\widehat{wv}=\frac{wv}{|wv|_1}.
$$
In particular, this action is well-defined on the normalised dual Tits cone introduced below.

\subsection{The dual Tits cone and the imaginary cone}

We use the following realisation of the dual of the Tits cone:
$$
U^*:=
\bigcap_{w\in W}
w\bigl(\PLC(\Pi)\cup\{0\}\bigr).
$$
It follows directly from this description that $U^*$ is a closed,
convex, $W$-invariant cone contained in
$\PLC(\Pi)\cup\{0\}$. We write
$$
\widehat{U^*}:=U^*\cap V_1.
$$
Since $W$ preserves $U^*$, the normalised action of $W$
$$
w\cdot\eta=\widehat{w\eta}
$$
is well-defined for every $w\in W$ and $\eta\in\widehat{U^*}$.

The fundamental domain of the  imaginary cone is
$$
\mathscr K:=
\left\{
 \eta\in\PLC(\Pi)\cup\{0\}
 \,\middle|\,
 (\eta,a)\leq0\text{ for every }a\in\Pi
\right\},
$$
and the imaginary cone is
$$
\mathscr Z:=\bigcup_{w\in W}w\mathscr K.
$$
We have
$$
\mathscr Z\subseteq U^*.
$$
A standard characterisation that will be used below is
$$
\eta\in\mathscr Z
\quad\Longleftrightarrow\quad
|\pos(\eta)|<\infty
\qquad(\eta\in U^*);
$$
see, for example, \cite[Proposition 4.22]{FU2}, \cite{Dyer12} and \cite{DHR13}.

Set
$$
Z:=\mathscr Z\cap V_1
$$
and define the \emph{imaginary convex body} by
$$
\overline Z:=\operatorname{cl}_{V_1}(Z)
            =\overline{\mathscr Z}^{\,V}\cap V_1,
$$
where the first closure is taken in $V_1$ and the second in $V$.
The basic relations among the objects introduced above are
$$
E\subseteq\overline Z
 \subseteq\widehat{U^*}
 \subseteq\conv(\Pi).
$$
Moreover, by \cite[Theorem~2.2]{DHR13},
$$
\overline Z=\conv(E).
$$
In particular, $\overline Z$ is compact, convex and invariant under the normalised action of $W$.

The following result links dominance and the dual of the Tits cone (\cite[Proposition 4.10]{FU2}):
\begin{proposition}
\label{domU}
For $x, y\in \Phi$, 
$$x\dom y \iff x-y\in U^*.$$
\qed
\end{proposition}

\begin{lemma}
\label{tranPos}
Let $v\in V$ and let $a\in \Pi$. Then 
\begin{equation*}
\pos(r_a v) =
\begin{cases} r_a(\pos(v)\setminus\{a\}) & \text{if } (v, a) > 0,\\
r_a\pos(v)=\pos(v) & \text{if } (v, a) = 0,\\
r_a \pos(v)\cup \{a\} & \text{if } (v, a) < 0.
\end{cases}
\end{equation*}
\end{lemma}
\begin{proof}
Let $x\in \Phi^+\setminus\{a\}$. Then $r_a x\in \Phi^+$ and so $x\in \pos(r_a v)$ if and only if $r_a x\in \pos(v)$. Thus 
$$\pos(r_a v)\setminus\{a\} =r_a(\pos(v)\setminus\{a\}).$$
If $(v, a)>0$ then $a\in \pos(v)$ and $a\notin\pos(r_a v)$, and hence 
$$\pos(r_a v)=r_a(\pos(v)\setminus\{a\}).$$
If $(v, a)=0$ then $r_a v=v$, and $a\notin\pos(v)=\pos(r_a v)$, and hence
$$\pos(v)=\pos(r_a v)=r_a \pos(v).$$
If $(v, a)<0$ then $a\notin\pos(v)$ and $a\in \pos(r_a v)$, and hence 
$$\pos(r_a v) = r_a \pos(v)\cup \{a\}.$$
\end{proof}

\begin{lemma}
\label{pos&conn}
Let $v\in V$ be such that $\pos(v)=\emptyset$.
\begin{itemize}
\item[(i)] $wv-v\in \PLC(\Pi)\cup \{0\}$,\,  $\forall w\in W$. Moreover, $\pos(wv)=\emptyset$ if and only if $wv=v$ which, in turn, occurs if and only if $w\in \langle\, r_a\mid (a, v)=0\,\rangle$.
\item[(ii)] Under the further assumption of $v\in \PLC(\Pi)$, if $\supp(v)$ is connected, then $\supp(wv)$ is connected,  $\forall w\in W$.
\end{itemize}
\end{lemma}
\begin{proof}
(i): Induction on $\ell(w)$. If $\ell(w)=0$ then the statement is trivially true. Next, suppose that $\ell(w)\geq 1$, and take $a\in \Pi$ with 
$\ell(w r_a)=\ell(w)-1$. Then $wa\in \Phi^-$, and therefore 
\begin{align*}
wv-v&=(w r_a)(r_a v)-v\\
    &=(w r_a)(v-2(v, a)a)-v\\
    &=(w r_a v-v )-2(v, a) w r_a a\\
    &=\underbrace{(w r_a v-v)}_\text{$\in\PLC(\Pi)\cup\{0\}$ by inductive hypothesis}+ 2\underbrace{(v, a)}_\text{$\leq 0$} \underbrace{w a}_\text{$\in \Phi^-$}\\
    &\in \PLC(\Pi)\cup \{0\}.
\end{align*}
Hence $wv-v\in \PLC(\Pi)\cup\{0\}$, for all $w\in W$.

Next, note that $\forall w\in W$ and $\forall a\in \Pi$
$$(wv, wa)=(v, a)\leq 0.$$
Now let $w\in W$ and $a\in \Pi$ with $\pos(wv)=\emptyset$ and $\ell(w r_a)=\ell(w)-1$. Suppose that 
$(v, a)<0$. Then $(wv, -wa)>0$ and $-wa\in \pos(wv)$, a contradiction. Thus, it also follows from an induction on $\ell(w)$ that $\pos(wv)=\emptyset$ if and only if $wv=v$ which, in turn, occurs if and only if $w\in \langle\, r_a\mid (a, v)=0\,\rangle$.

(ii): Induction on $\ell(w)$. If $\ell(w)=0$ then $wv =v$ and there is nothing to prove. Next suppose that $\ell(w)\geq 1$, and take $a\in \Pi$ with $\ell(w r_a)=\ell(w)-1$. Then $wa\in \Phi^-$. Note that $(v, a)\leq 0$, and if $(v, a)=0$ then 
$$w v = wr_a (r_a v) =wr_a v.$$
Since $\ell(w r_a) =\ell(w)-1$ the desired result follows by induction. Otherwise $(v, a)<0$, and then 
\begin{equation*}
w v = wr_a (r_a v)=w r_a v-2(a, v)w r_a a =\underbrace{w r_a v}_\text{$\in \PLC(\Pi)$}+ 2\underbrace{(a, v)}_\text{$<0$} \underbrace{wa}_\text{$\in \Phi^-$}.
\end{equation*}
Whence 
\begin{equation}
\label{s}
\supp(wv)=\supp(wr_a v)\cup \supp(wa).
\end{equation}
Since $wa\in \Phi$, it follows that $\supp(wa)$ is connected, and moreover, $\supp(w r_a v)$ is connected, by the inductive hypothesis. Now note that 
$$0> (w r_a v, wr_a a)=-(wr_av, wa),$$ 
and hence there exists some $b\in \supp(w r_a v)$ and $c\in \supp(wa)$ with $(b, c)\neq 0$. This observation together with (\ref{s}) and the connectedness of $\supp(w r_a v)$ and $\supp(wa)$ readily yields that $\supp(wv)$ is connected, completing the induction. 
\end{proof}


The following is a key geometric property enjoyed by elements of the dual of the Tits cone $U^*$.
This is well-known to the experts of the field (and it is a special case of \cite[Theorem 5.8]{FU3}), but we include a proof here for completeness.  
\begin{proposition}
\label{leq0}
Let $u,v\in U^*$ be arbitrary. Then 
$$(u, v)\leq 0.$$
\end{proposition}
\begin{proof}
Suppose, for a contradiction, that there are $u, v\in U^*$ with $(u, v)>0$. Replacing $v$ by any positive scalar multiple of itself, we may assume that $(u, v)=1$. Set $h:=|u|_1$, and define
$$\mathscr{A}:=\Big\{x\in U^*\colon |x|_1\leq |v|_1 \text{ and $\exists z\in U^*$ with $(z, x)\geq 1$ and $|z|_1\leq h$} \Big\}.$$
Note that $v\in \mathscr{A}$ and hence $\mathscr{A}\neq \emptyset$. Next, set 
$$\epsilon:=\frac{2}{h|\Pi|}>0.$$
Claim: For any given $x\in \mathscr{A}$ we can find $y\in \mathscr{A}$ with 
$$|y|_1\leq |x|_1-\epsilon.$$
Note that once this claim is proved, starting with $x=v$, a finite number of iterations of the claimed process will produce some $y\in \mathscr{A}$ with $|y|_1<0$, contradicting 
$$y\in\mathscr{A}\subseteq U^*= \bigcap_{w\in W}w \big(\PLC(\Pi)\cup \{0\}\big)\subset \PLC(\Pi)\cup \{0\}.$$
Now we prove the claim: given $x\in \mathscr{A}$, choose $z\in U^*$ with $(z, x)\geq 1$ and $|z|_1\leq h$.
Write $z=\sum_{b\in \Pi}\coeff_b(z)b$, then 
$$(z, x)=\sum_{b\in \Pi}\coeff_b(z)(b, x)\geq 1.$$
This implies that $\exists b\in \Pi$ with
$$\coeff_b(z)(b, x)\geq \frac{1}{|\Pi|},$$
and since $\coeff_b(z)\leq |z|_1\leq h$, hence 
$$(b, x)\geq \frac{1}{h|\Pi|}=\frac{\epsilon}{2}.$$
Put $y:=r_b x= x-2(b, x) b$, and note that $y\in U^*$ for $x\in U^*$ and $U^*$ is $W$-invariant. Then 
$$|y|_1=|x|_1-2(b, x)\leq |x|_1-\epsilon.$$
To complete the proof of the claim we need to prove that $y\in \mathscr{A}$, and it only remains to find some $t\in U^*$ with 
$(t, y)\geq 1$ and $|t|_1\leq h$. We have two possibilities to cover: either $(b, z)\geq 0$ or else $(b, z)<0$, and we show that in either case we can find such a suitable $t$.

If $(b, z)\geq 0$, then put $t:=r_b z$. Then $t\in U^*$, and 
$$(t, y)=(r_b z, r_b x)=(z, x)\geq 1,$$
and 
$$|t|_1=|z|_1-2(b, z)\leq |z|_1 \leq h.$$

If $(b, z)<0$, then put $t:=z\in U^*$. Then $|t|_1=|z|_1\leq h$ and
$$(z, y)=(z, r_b x)=(z, x-2(b, x)b)=(z, x)-2\underbrace{(b, x)}_\text{$>0$}\underbrace{(z, b)}_\text{$<0$}\geq (z, x)\geq 1.$$

\end{proof}

\begin{lemma}
\label{neg}
Let $\eta\in U^*\setminus\{0\}$, and let $x\in \Phi^+$ with $\supp(x)\not\subset\supp(\eta)$. Then $(x, \eta)\leq 0$.
\end{lemma} 
\begin{proof}
Suppose for a contradiction that $(x, \eta)>0$. Choose $a\in \Pi$ such that $a\in \supp(x)$ but $a\notin\supp(\eta)$.
Since $\eta\in U^*$, it follows that
$$r_x \eta=\eta-2(x, \eta)x\in \PLC(\Pi).$$
But, under the assumption that $\Pi$ is a basis for $V$, the coefficient for $a$ in $r_x\eta$ is a strictly negative quantity, namely, $-2(x, \eta)\coeff_a(x)$, a contradiction.
\end{proof}

\begin{proposition}
  \label{epsilon}
  Let $\eta\in U^*$. Then for each $\epsilon >0$, the following set 
  $$\{\, x\in \Phi^+\mid (x, \eta)\geq \epsilon \,\}$$
  is finite.
\end{proposition}
\begin{proof}
Suppose the contrary holds. Then there exists an infinite injective sequence of positive roots $(x_i)_{i\in \N}\subset \Phi^+$ with 
$(\eta, x_i)\geq \epsilon$ for all $i\in \N$. On the other hand, since $\eta\in U^*$, it follows that
\begin{align*}
0<|r_{x_i}\eta|_1 &=|\eta|_1-2(\eta, x_i)|x_i|_1\\
                  &\leq |\eta|_1-2\epsilon |x_i|_1.
\end{align*}
Since the sequence $(x_i)_{i\in \N}$ is infinite and injective, it follows from Lemma~\ref{normbd} that the sequence of real 
numbers $(|x_i|_1)_{i\in \N}$ tends to infinity. Hence there exists some $n\in \N$ such that 
$$|x_i|_1>\frac{|\eta|_1}{2\epsilon},\quad \forall i\geq n.$$
Consequently, $|r_{x_i}\eta|_1<0$ for all $i\geq n$, contradicting the assumption of $\eta\in U^*=\bigcap_{w\in W} w(\PLC(\Pi)\cup\{0\})$.
\end{proof}
Note that Proposition \ref{epsilon} has the following immediate consequence:
\begin{corollary}
  \label{lim0}
  Let $\eta\in U^*$, and let $(x_i)_{i\in \N}\subset \Phi^+$ be an infinite injective sequence of positive roots such that $(\eta, x_i)\geq 0$ for all $i\in \N$. Then 
  $$\lim_{i\to\infty}(\eta, x_i)=0.$$
  \qed
\end{corollary}

The next result is an easy observation on the dual of the Tits cone, and it can be useful for general purpose.
\begin{lemma}
\label{two}
Suppose that $\eta\in U^*$, and $a, b\in \Pi\cap \pos(\eta)$ with $a\neq b$. Then 
$$(a, r_b \eta)\geq (a, \eta).$$
In particular, $a\in \pos(r_b \eta)$.
\end{lemma}
\begin{proof}
\begin{align*}
(a, r_b\eta)&=(a, \eta-2(\eta, b)b)=(a, \eta)-2\underbrace{(\eta, b)}_\text{$>0$}\underbrace{(a, b)}_\text{$\leq 0$}\\
            &\geq (a, \eta).
\end{align*}
\end{proof}

We close this section with the following important property of Coxeter groups of finite rank (originally due to Krammer \cite{DK94}, see also \cite[Proposition 4.5.5]{ABFB}):
\begin{proposition}
\label{lowbd}
If $(W, R)$ is a Coxeter system in which $R$ is a finite generating set, then the following set of possible bilinear form values in the range of $[-1, 1]$ between pairs of roots 
$$\big\{ (a, b)\mid \text{$a, b\in \Phi$ with $-1\leq (a, b) \leq 1$} \big\}$$
is finite. In particular, there exists some fixed $L>0$ such that 
$$|(a, b)|\geq L, \text{ whenever $a, b\in \Phi$ with $(a, b)\neq 0$}.$$ 
\qed
\end{proposition}

%
%
%

\section{Asymptotic action on the dual of the Tits cone}
\label{sec:asymptotic-action}
\begin{definition}
Let $\eta\in U^*$ be such that $\pos(\eta)\neq \emptyset$. We construct inductively a sequence $\eta_1,\eta_2, \ldots$, of elements of $U^*$ by
\begin{align*}
\eta_1&:=\eta,\\
\noalign{\hbox{and for all $i\in \N$}}
\eta_{i+1} &:=r_{a_i} \eta_i,\quad \text{where $a_i\in \Pi\cap \pos(\eta_i)$, whenever $\pos(\eta_i)\neq \emptyset$.}
\end{align*}
Let $\mathbf{a}:=(a_1, a_2, \ldots)$ be the associated sequence of simple roots, and call $\mathbf{a}=(a_i)_{1\leq i\leq |\pos(\eta)|}$ a \emph{transitional sequence} for $\eta$, and we call the sequence $(\eta_i)_{1\leq i\leq|\pos(\eta)|+1} \subset U^*$ the \emph{descending sequence} corresponding to the transitional sequence $\mathbf{a}$. 
\end{definition}

\begin{remark}
  \label{ts}
  If $\eta\in U^*$ is such that $|\pos(\eta)|<\infty$ (that is, $\eta\in \mathscr{Z}=\bigcup_{w\in W} w\mathscr{K}$) then Lemma~\ref{tranPos} implies that 
  $$\pos(\eta_{|\pos(\eta)|+1})=\emptyset,$$
  and consequently, every transitional sequence terminates, and every transitional sequence contains $|\pos(\eta)|$ terms with the corresponding descending sequence containing $|\pos(\eta)|+1$ terms. 
  
  If $\eta\in U^*$ is such that $|\pos(\eta)|=\infty$ (that is, $\eta\in U^*\setminus\mathscr{Z}$), then $\forall w\in W$
  $$|\pos(w\eta)|=\infty,$$
  and, in particular, 
  $$|\pos(\eta_i)|=\infty, \quad \forall i\in \N,$$
  and consequently, every transitional sequence for such an $\eta$ must be an infinite sequence, and we may write $(a_i)_{i\in \N}$  for such an infinite transitional sequence.
\end{remark}

\begin{remark}
\label{lim}
Let $\eta\in U^*$ be such that $|\pos(\eta)|=\infty$, and let $\mathbf{a}=(a_i)_{i\in \N}$ be a transitional sequence for $\eta$ with corresponding $(\eta_i)_{i\in \N}$. Then $|\pos(\eta_i)|=\infty$ for all $i\in \N$. 

Note that $\eta_i\in U^*=\bigcap_{w\in W} w (\PLC(\Pi)\cup\{0\} )\subset \PLC(\Pi)\cup\{0\}$, and therefore
$$0\leq \coeff_a(\eta_i), \quad \forall a\in \Pi, \,\forall i\in \N,$$
and since $a_i\in \pos(\eta_i)\cap \Pi$, it follows that 
\begin{align*}
\coeff_{a_i}(\eta_{i+1})&<\coeff_{a_i}(\eta_i),\quad \forall i\in \N,\\
\noalign{\hbox{while}}
\coeff_a(\eta_{i+1}) &=\coeff_a(\eta_i),\quad \forall a\in \Pi\setminus\{a_i\}, \text{ and } \forall i\in \N.
\end{align*}
Consequently, 
$$0\leq \coeff_a(\eta_{i+1})\leq \coeff_a(\eta_i),\quad \forall i\in \N, \text{ and } \forall a\in \Pi.$$
Therefore for each $a\in \Pi$ the sequence $(\coeff_a(\eta_i))_{i\in \N}$ is a (weakly) decreasing sequence of non-negative real numbers, and thus
$$\lim_{i\to\infty}\coeff_a(\eta_i) \text{  exists for each $a\in \Pi$},$$
which, in turn, establishes that the following limit exists: 
$$\eta_{\mathbf{a}}:=\lim_{i\to\infty}\eta_i=\sum_{a\in \Pi}(\lim_{i\to\infty} \coeff_a(\eta_i)a).$$
Note that the sequence of real numbers $(|\eta_i|_1)_{i\in \N}$ is  (weakly) decreasing, we are justified in calling $(\eta_i)_{1\leq i\leq |\pos(\eta)|+1}$ a descending sequence. 
\end{remark}

\begin{definition}
Let $\eta\in U^*$ with $|\pos(\eta)|=\infty$, and let $\mathbf{a}=(a_i)_{i\in \N}$ be an infinite transitional sequence for $\eta$ giving rise to $(\eta_i)_{i\in\N}$. Denote
$$\eta_{\mathbf{a}}:=\lim_{i\to\infty}\eta_i,$$
and call $\eta_{\mathbf{a}}$ \emph{the limit of $\eta$ via the transitional sequence $\mathbf{a}$}.

Call an index $m_{\mathbf{a}}\in \N$ a \emph{separating index} for $\mathbf{a}=(a_i)_{i\in \N}$ if from the $m_{\mathbf{a}}$-th term of 
$\mathbf{a}=(a_i)_{i\in \N}$ onward, each $a\in \Pi$ either occurs infinitely many times in $(a_i)_{i\geq m_{\mathbf{a}}}$, that is, 
$$|\{i\geq m_{\mathbf{a}}\colon a=a_i\}|=\infty;$$
or else does not occur at all, that is, 
$$|\{i\geq m_{\mathbf{a}}\colon a=a_i\}|=0.$$

Moreover, we define the associated $\mathbf{a}_{\infty}$ to be
$$\mathbf{a}_{\infty}:=\big\{\,a\in \Pi\mid\text{ $a$ occurs infinitely many times in $\mathbf{a}=(a_i)_{i\in \N}$}\,\big\}.$$

On the other hand, when $\eta\in U^*$ with $|\pos(\eta)|<\infty$, we let $\mathbf{a}=(a_i)_{1\leq i\leq |\pos(\eta)|}$ be a finite transitional sequence for $\eta$, giving rise to $(\eta_i)_{1\leq i\leq |\pos(\eta)|+1}$. Denote
$$\eta_{\mathbf{a}}:=\eta_{|\pos(\eta)|+1},$$
and call $\eta_{\mathbf{a}}$ \emph{the limit of $\eta$ via the transitional sequence $\mathbf{a}$}. In this case, we take $m_{\mathbf{a}}:=|\pos(\eta)|+1$, and call $m_{\mathbf{a}}$ the separating index for $\mathbf{a}$. Moreover the limit of $\eta$ via $\mathbf{a}$ is then
$$\eta_{\mathbf{a}}=\eta_{m_{\mathbf{a}}}.$$
In the special case of $\pos(\eta)=\emptyset$, we set the convention for $\eta$ to have only the empty transitional sequence with corresponding limit $\eta$ and separating index to be simply $1$.  
\end{definition}

\begin{remark}
In the situation described in the above definition,  when $|\pos(\eta)|=\infty$, note that if $m_{\mathbf{a}}$ is a separating index for $\mathbf{a}=(a_i)_{i\in \N}$, then every $m\in \N$ with $m\geq m_{\mathbf{a}}$ is also a separating index for $\mathbf{a}$.

Moreover, $\mathbf{a}_{\infty}$ can be alternatively characterised as 
\begin{equation}
\label{alt}
  \mathbf{a}_{\infty}=\big\{\, a_i\mid i\geq m_{\mathbf{a}} \,\big\}.
\end{equation}
\end{remark}

\begin{definition}
Let $\eta\in U^*$. Define $T(\eta)$ to be the set of all transitional sequences for $\eta$, that is
\begin{align*}
&T(\eta):=\{(a_i)_{1\leq i\leq |\pos(\eta)|}\mid\!\text{$(a_i)_{1\leq i\leq |\pos(\eta)|}$ is a transitional sequence of $\eta$}\}\\
       &=\{(a_i)_{1\leq i \leq |\pos(\eta)|} \mid \text{$a_i\in \Pi$, $(r_{a_{i-1}} \cdots r_{a_1}\eta, a_i )>0,\,1\leq i\leq |\pos(\eta)|$ }\}.
\end{align*}
\end{definition}

\begin{remark}
\label{eg1}
Suppose that $\eta \in U^*\setminus \{0\}$, and set $$h:=\inf_{w\in W} |w\eta|_1.$$ 
If $w\in W$ with $|w\eta|_1\neq h$ then $\pos(w\eta) \neq \emptyset$. Indeed, there must be some 
$a\in\Pi$ with $(a, w\eta)>0$. Otherwise Lemma~\ref{pos&conn} yields that for all $w'\in W$
$$w'w\eta-w\eta \in \PLC(\Pi)\cup\{0\},$$
and therefore
$$|w'w\eta|_1\geq |w\eta|_1>h, \quad \forall w'\in W,$$
contradicting the definition of $h$. Assuming that there is no $\eta'$ in the $W$-orbit of $\eta$ with $|\eta'|_1=h$, then any transitional sequence for $\eta$ must be infinite, and $|\pos(\eta)|=\infty$.
\end{remark}

\begin{lemma}
\label{eg2}
Suppose that $\eta\in U^*$ with $|\pos(\eta)|=\infty$. Let $\mathbf{a}:=(a_i)_{i\in \N}$ be an infinite transitional sequence for $\eta$ with the corresponding descending sequence $(\eta_i)_{i\in \N}$ where
\begin{align*}
  \eta_1 &:= \eta, \\
\noalign{\hbox{and}}  
  \eta_{i+1} &:= r_{a_i} \eta_i,
\end{align*} 
with the extra condition that  $a_i\in \Pi\cap \pos(\eta_i)$ such that 
$$(a_i, \eta_i)\geq (a, \eta_i), \quad \forall a\in \Pi\cap\pos(\eta_i).$$
Then for the corresponding limit $\eta_{\mathbf{a}}$ we have that 
$$\pos(\eta_{\mathbf{a}})=\emptyset.$$
\end{lemma}
\begin{proof}
Suppose, for a contradiction, that $(a, \eta_{\mathbf{a}})>0$ for some $a\in \Pi$.
Since $$\lim_{i\to\infty} (a, \eta_i) =(a, \eta_{\mathbf{a}}),$$
and 
$$\lim_{i\to\infty} |\eta_i|_1 =|\eta_{\mathbf{a}}|_1,$$
we may choose $i$ such that 
\begin{align*}
  |\eta_i|_1 &< |\eta_{\mathbf{a}}|_1+(a, \eta_{\mathbf{a}})\\
  \noalign{\hbox{and}}
  (a, \eta_i) &>\frac{1}{2}(a, \eta_{\mathbf{a}}). 
\end{align*}
Now 
$$(a_i, \eta_i)\geq (a, \eta_i)>\frac{1}{2}(a, \eta_{\mathbf{a}}),$$
and $\eta_{i+1}=\eta_i-2(a_i, \eta_i)a_i$, we see that 
\begin{align*}
  |\eta_{i+1}|_1  =|\eta_i|_1 -2(a_i, \eta_i) 
                 &\leq |\eta_i|_1 -2(a, \eta_i)\\ 
                 &\leq |\eta_i|_1 -(a, \eta_{\mathbf{a}}) \\
                 & < |\eta_{\mathbf{a}}|_1,
\end{align*}
a contradiction, since $|\eta_{\mathbf{a}}|_1$ is the limit of the decreasing sequence $(|\eta_i|_1)_{i\in \N}$.
\end{proof}

\begin{remark}
Suppose that $\eta\in U^*$ with $|\pos(\eta)|=\infty$. Suppose that we impose a fixed ordering on the elements of $\Pi$, and we construct a transitional sequence
$\mathbf{a}:=(a_i)_{i\in \N}$ for $\eta$ with respect to this ordering, cyclically. That is, starting with $\eta_1:=\eta$, and we choose $a_i\in \Pi\cap\pos(\eta_i)$ via the following scheme: scan the elements of $\Pi$ under this chosen ordering, until the first $a\in \Pi$ such that $a\in \pos(\eta_1)$, and set $a_1=a$, and $\eta_2=r_{a_1}\eta_1$.
Next, continue the scanning  process with $\eta_2$ and the element in $\Pi$ immediately succeeding $a_1$ under the chosen ordering, then either there is a subsequent element  $a'\in\Pi\cap \pos(\eta_2)$ under this chosen ordering, or else we reach the last element of $\Pi$ under this ordering without finding another element  $a'\in\Pi\cap \pos(\eta_2)$, and in the latter situation we (cyclically) restart the scanning process from the first element of $\Pi$ under this chosen order, and repeat the process so far. Then 
$$\pos(\eta_{\mathbf{a}})=\emptyset.$$
Indeed, suppose, for a contradiction, that $\pos(\eta_{\mathbf{a}})\neq \emptyset$ and we may choose $a\in \Pi\cap\pos(\eta_{\mathbf{a}})$. That is, $\lim_{i\to\infty}(a, \eta_i)>0$.
Thus, for some fixed $\epsilon >0$ there exists some $N\in \N$ with 
\begin{equation}
\label{cyclic}
(a, \eta_i)\geq \epsilon, \quad \forall i\geq N.
\end{equation}
Then Lemma~\ref{two}, the finiteness of $\Pi$ and the cyclic and sequential scheme of choosing the transitional sequence $\mathbf{a}=(a_i)_{i\in\N}$ adopted here together imply that
$$a=a_j,\quad \text{for infinitely many $j\in \N$}.$$
Thus (\ref{cyclic}) implies that for some $k$ sufficiently large
$$\coeff_a{\eta_k}<0,$$
a contradiction to $w\eta\in U^*\subset \PLC(\Pi)\cup\{0\}$. 
\end{remark}

\begin{lemma}
\label{trivial}
Suppose that $\eta\in U^*$ with $\pos(\eta)$ being infinite, and let $\mathbf{a}=(a_i)_{i\in \N}$ be a transitional sequence for $\eta$ with corresponding  $\mathbf{a}_{\infty}$.
Then $$\mathbf{a}_{\infty}\subseteq \supp(\eta_i),\,\,\forall i\in \N.$$
\end{lemma}
\begin{proof}
Suppose that there exists $a\in \mathbf{a}_{\infty}$ with $a\notin \supp(\eta_j)$ for some $j\in \N$. 
Since $\forall k\geq j$ and $\forall a\in \Pi$,
 $$\coeff_a(\eta_k)\leq \coeff_a(\eta_j),$$ 
it follows that 
\begin{equation}
\label{dec}
\supp(\eta_k)\subseteq \supp(\eta_j), \,\,\forall k\geq j.
\end{equation}
Thus $a\notin \supp(\eta_k)$, for all $k\geq j$, and this means that $(a, \eta_k)\leq 0$ for all $k\geq j$, contradicting $a\in \mathbf{a}_{\infty}$.
\end{proof}

\begin{lemma}
\label{easy}
Suppose that $\eta\in U^*$ with $\pos(\eta)$ being infinite, and let $\mathbf{a}=(a_i)_{i\in\N}$ be a transitional sequence for $\eta$ with  corresponding separating index $m_{\mathbf{a}}$ and $\mathbf{a}_{\infty}$.  Then for each $a\in \supp(\eta_{m_{\mathbf{a}}}) \setminus \mathbf{a}_{\infty}$, 
$$\coeff_a (\eta_i) =\coeff_a(\eta_{m_{\mathbf{a}}}),\,\,\forall i\geq m_{\mathbf{a}}.$$
\end{lemma}
\begin{proof}
If $a\notin \mathbf{a}_{\infty}$ then it follows from the alternative characterisation (\ref{alt}) of $\mathbf{a}_{\infty}$ that 
$$|\{i\geq m_{\mathbf{a}}\colon a=a_i\}|=0,$$
and consequently $a$ shall not appear as a term of $(a_i)_{i\geq m_{\mathbf{a}}}$, and therefore 
$\coeff_a (\eta_i) =\coeff_a(\eta_{m_{\mathbf{a}}})$, $\forall i\geq m_{\mathbf{a}}$.
\end{proof}

\begin{lemma}
\label{stn}
Suppose that $\eta\in U^*$ with $\pos(\eta)$ being infinite, and let $\mathbf{a}=(a_i)_{i\in \N}$ be a transitional sequence for $\eta$ with  corresponding descending sequence $(\eta_i)_{i\in \N}$, separating index $m_{\mathbf{a}}$, and $\mathbf{a}_{\infty}$. Then the support of the $\eta_i$ will stabilise, and, in fact 
$$\supp(\eta_n)=\supp(\eta_{m_{\mathbf{a}}}), \,\,\forall n\geq m_{\mathbf{a}}.$$
\end{lemma}
\begin{proof}
Note that from (\ref{dec}) we see that $(\supp(\eta_i))_{i\in \N}$ form a (weakly) decreasing sequence of subsets of the finite set $\Pi$ with respect to inclusion, bounded from below, and therefore there exists some $N\in \N$  with 
$$\supp(\eta_i)=\supp(\eta_N),\,\,\forall i\geq N.$$
Next we show that $m_{\mathbf{a}}$ qualifies to be such an $N$, and in view of (\ref{dec}) we only need to show that 
$$\supp(\eta_{m_{\mathbf{a}}})\subseteq \supp(\eta_n), \,\forall n\geq m_{\mathbf{a}}.$$
Note that Lemma~\ref{easy} yields that 
$$\supp(\eta_{m_{\mathbf{a}}})\setminus \mathbf{a}_{\infty}\subseteq \supp(\eta_n),\,\,\forall n\geq m_{\mathbf{a}},$$
this together with Lemma~\ref{trivial} then establishes that 
$$\supp(\eta_{m_{\mathbf{a}}})\subseteq \supp(\eta_n),\,\,\forall n\geq m_{\mathbf{a}}.$$
\end{proof}



\begin{lemma}
\label{pos}
Let $\eta\in U^*$ with $\pos(\eta)$ being non-empty, and let $\eta_{\mathbf{a}}$ be the limit of $\eta$ via some transitional sequence  $\mathbf{a}=(a_i)_{1\leq i \leq |\pos(\eta)|}$. For all $j<i\leq |\pos(\eta)|$, define
$a_{i;j}:=r_{a_{j}} r_{a_{j+1}}\cdots r_{a_{i-1}} a_i$. Then for all $j<i\leq |\pos(\eta)|$ 
\begin{itemize}
\item[(i)] $a_{i;j}\in \pos(\eta_j)$;
\item[(ii)] $\ell(r_{a_i}\cdots r_{a_j})=i-j+1$, and correspondingly 
 $$N(r_{a_i}\cdots r_{a_j})=\{a_j,a_{j+1; j}, a_{j+2; j}, \ldots, a_{i;j}\}.$$ 
 Moreover, $a_{k;j}\neq a_{l;j}$ whenever $k$ and $l$ are distinct integers both greater than $j$. 
\end{itemize}
\end{lemma} 
\begin{proof}
For (i): By the definition of $\mathbf{a}=(a_i)_{1\leq i\leq |\pos(\eta)|}$, for each $j<i\leq |\pos(\eta)|$
\begin{align*}
0<(a_i, \eta_i)&=(a_i, r_{a_{i-1}}\cdots r_{a_j}\eta_j)\\
               &=(r_{a_j}\cdots r_{a_{i-1}} a_i, \eta_j)\\
               &=(a_{i;j},\eta_j).          
\end{align*}
Thus for (i), it is enough to show that $a_{i;j}\in \Phi^+$. Suppose for a contradiction that $a_{i;j}\in \Phi^-$. Then there exists some
index $k\in \{j, \ldots, i-1\}$ with 
\begin{align*}
r_{a_{k+1}}\cdots r_{a_{i-1}} a_i &\in \Phi^+,\\
\noalign{\hbox{and yet}}
r_{a_k}r_{a_{k+1}}\cdots r_{a_{i-1}} a_i &\in \Phi^-.
\end{align*}
Since $a_k\in \Pi$, it follows from the above that $r_{a_{k+1}}\cdots r_{a_{i-1}} a_i =a_k$. Now by our definition of 
$\mathbf{a}=(a_i)_{1\leq i\leq |\pos(\eta)|}$:
\begin{align*}
0>-(a_k, \eta_k)&=(r_{a_k} a_k\,, \,r_{a_{k-1}}\cdots r_{a_j}\eta_j)\\
               &=(r_{a_k} r_{a_{k+1}}\cdots r_{a_{i-1}} a_i\,,\,r_{a_{k-1}}\cdots r_{a_j}\eta_j)\\
               &=(a_i, \,r_{a_{i-1}}\cdots r_{a_{j}} \eta_j )\\
               &=(a_i, \eta_i)>0,
\end{align*}
a contradiction. 

For (ii): By (i) above, for all $j<i\leq |\pos(\eta)|$, we have that 
$$a_{i;j}=r_{a_{j}}\cdots r_{a_{i-1}}a_i \in \Phi^+, $$
thus $\ell(r_{a_j}\cdots r_{a_{i-1}} r_{a_i})=\ell(r_{a_j}\cdots r_{a_{i-1}})+1$. Consequently, an induction yields that for all $j<i \leq |\pos(\eta)|$
$$\ell(r_{a_j}\cdots r_{a_i})=i-j+1.$$ 
Note that for each $k\in\{j, \ldots, i\}$, since $r_{a_i}\cdots r_{a_{k+1}}$ is a left-segment of $r_{a_i}\cdots r_{a_{k+1}}r_{a_k}$, which is in turn a left segment of a reduced expression $r_{a_i}\cdots r_{a_j}$, it follows that $\ell(r_{a_i}\cdots r_{a_{k+1}}r_{a_k})=\ell(r_{a_i}\cdots r_{a_{k+1}})+1$, and therefore $r_{a_i}\cdots r_{a_{k+1}}a_k\in\Phi^+$. Then
\begin{align*}
r_{a_i}\cdots r_{a_j} a_{k;j}&=r_{a_i}\cdots r_{a_j} r_{a_j}\cdots r_{a_{k-1}}a_k\\
                             &=r_{a_i}\cdots r_{a_k} a_k\\
                             &=-r_{a_i}\cdots r_{a_{k+1}}a_k\in\Phi^-,
\end{align*}
whence $a_{k;j}\in N(r_{a_i}\cdots r_{a_j})$ for each $k\in\{j, \ldots, i\}$.
On the other hand, it follows from an induction on $\ell(r_{a_i}\cdots r_{a_j})$ that 
$$N(r_{a_i}\cdots r_{a_j})\subseteq \{a_j,a_{j+1; j}, a_{j+2; j}, \ldots, a_{i;j}\}.$$ 
Then it follows from $\ell(r_{a_i}\cdots r_{a_j})=i-j+1$ that 
$a_{k;j}\neq a_{l;j}$ whenever $k$ and $l$ are distinct integers in $\{j, \ldots, i\}$.
\end{proof}

\begin{lemma}
\label{form0}
Let $v\in V$, and define an infinite sequence $(v_i)_{i\in \N}$ by:
$$v_1:=v, \,\text{and}\,\, v_{i+1}:=r_{a_i} v_i ,\,\text{where $a_i\in \Pi$ for all  $i\in \N$}.$$
Suppose that $v_{\infty}:=\underset{i\to\infty}{\lim} v_i$ exists. Then 
$$(a, v_{\infty})=0,$$
for all those  $a\in \Pi$ with $a=a_i$ for infinitely many $i\in\N$.
\end{lemma}
\begin{proof}
Suppose for a contradiction that $(a, v_{\infty})>0$. Since $v_{\infty}=\underset{i\to\infty}{\lim} v_i$, it follows that there exists 
some $N\in \N$ with 
\begin{equation}
\label{v_infty}
(a, v_i)>0 \text{ for all $i\geq N$}.
\end{equation} 
Now given that $a=a_i$ for infinitely many $i\in \N$, there exists some $j\geq N$ with $a=a_j$. Then (\ref{v_infty}) yields that
$$(a, v_{j+1})=-(a, v_j)<0,$$
which is itself contradicting (\ref{v_infty}).
An entirely similar argument yields that $(a, v_{\infty})<0$ also leads to a contradiction, and therefore the desired result follows.
\end{proof}

\begin{remark}
\label{eta_a}
Note that for $\eta \in U^*$ with $|\pos(\eta)|=\infty$, let $\mathbf{a}=(a_i)_{i\in \N}$ be a transitional sequence for $\eta$ with corresponding 
$\eta_{\mathbf{a}}$ and $\mathbf{a}_{\infty}$. Then Lemma~\ref{form0} has the following immediate consequence:
$$\mathbf{a}_{\infty} \subseteq \{a\in \Pi\mid (a, \eta_{\mathbf{a}})=0\}.$$
\end{remark}

\begin{proposition}
\label{component}
Let $\eta\in U^*$ with $\pos(\eta)$ being infinite, and let $\eta_{\mathbf{a}}$ be the limit of $\eta$ via some transitional sequence  $\mathbf{a}=(a_i)_{i\in \N}$ with  $m_{\mathbf{a}}$ being a separating index. 
Then $\supp(\eta_{\mathbf{a}})$ is a union of connected components of $\supp(\eta_{m_{\mathbf{a}}})$.
\end{proposition}
\begin{proof}
First, recall that $\supp(\eta_{i+1})\subseteq \supp(\eta_i)$ for all $i\in \N$. Now suppose that $a\in \supp(\eta_{m_{\mathbf{a}}})\setminus \supp(\eta_{\mathbf{a}})$. 
Then the coefficients of $a$ would have decreased from $\eta_{m_{\mathbf{a}}}$  along the journey to the limit $\eta_{\mathbf{a}}$, and this means that  $a=a_i$ for infinitely many $i\in \N$. 
Then Lemma~\ref{form0} implies that $(a, \eta_{\mathbf{a}})=0$. On the other hand for all $a\in \Pi\setminus\supp(\eta_{\mathbf{a}})$ we have that $(a, b)\leq 0$ for each $b\in \supp(\eta_{\mathbf{a}})$.
Consequently, $(a, b)=0$ for all $a\in\supp(\eta_{m_{\mathbf{a}}})\setminus \supp(\eta_{\mathbf{a}})$ and $b\in\supp(\eta_{\mathbf{a}})$ and we are done.
\end{proof}

\begin{lemma}
\label{empty}
Let $\eta\in U^*$ with $|\pos(\eta)|=\infty$, and let $\eta_{\mathbf{a}}$ be the limit of $\eta$  arising from some transitional sequence  $\mathbf{a}=(a_i)_{i\in \N}$ with $m_{\mathbf{a}}$ being a separating index.
Let $J\subseteq \supp(\eta_{m_{\mathbf{a}}})$ be a union of connected components such that $|\pos(\eta_{m_{\mathbf{a}}}|_J)|<\infty$ (where $\cdot|_J$ means the projection/restriction of $\cdot$ onto the subspace spanned by $J$). Then $J\cap \mathbf{a}_{\infty}=\emptyset$.   
\end{lemma}
\begin{proof}
Suppose, for a contradiction, that $J\cap \mathbf{a}_{\infty}\neq \emptyset$ and we may choose $a\in J\cap \mathbf{a}_{\infty}$. Then by the 
definition of $\mathbf{a}=(a_i)_{i\in \N}$, 
\begin{equation}
\label{AJ}
(a, \eta_i|_J)>0 \quad\text{for infinitely many $i\in \N$.}
\end{equation} 
Since $|\pos(r_a x)|=|\pos(x)|-1$ for all $x\in V$ whenever $a\in \Pi\cap \pos(x)$, it follows from the definition of $J$ that there exists some $k\in \N$ with $\pos(\eta_k|_J)=\emptyset$, contradicting (\ref{AJ}).
\end{proof}

\begin{definition}
For $I\subset \Pi$, let $W_I:=\langle r_a\mid a\in I\rangle$ be the
\emph{standard parabolic subgroup} corresponding to $I$. Let
$$U^*_I:=\bigcap_{w\in W_I} w(\R_{\geq 0}I)$$
be the \emph{dual of the Tits cone} associated with the standard parabolic subgroup $W_I$. 
\end{definition}

\begin{lemma}
\label{dual}
$$U^*_I\subseteq U^*, \,\text{for all $I\subseteq \Pi$.}$$
\end{lemma}
\begin{proof}
Suppose, for a contradiction, that there exists some $x\in U^*_I$ and yet $x\notin U^*$. Note that, in particular, $x\neq 0$. Then there must exist some $w\in W\setminus W_I$ with $wx\notin \PLC(\Pi)$. Write $w=d w_I$ with $w_I\in W_I$ and $d\in W$ with the property that $da\in \Phi^+$ for all $a\in I$. Note that under this condition we have $\ell(w)=\ell(d)+\ell(w_I)$.
Now since $x\in U^*_I$, it follows that $w_I x\in \PLC(I)$, and thus
$$w_I x =\sum_{a\in I} \lambda_a a, \,\text{where $\lambda_a \geq 0$ for all $a\in I$.}$$
Then $wx =dw_I x= d(\underset{a\in I}{\sum} \lambda_a a)=\underset{a\in I}{\sum}\lambda_a d(a)\in \PLC(\Pi)$, a contradiction as required. 
\end{proof}

\begin{lemma}
\label{int}
Suppose that $x\in Q\cap U^*$, and suppose that $I_1, \ldots, I_n$ are the connected components of $\supp(x)$.
Write $x_i:=x|_{I_i}$ (the restriction of $x$ on the subspace $\R I_i$) for all $i\in \{1, \ldots, n\}$.  Then 
$$x_i\in U^*_{I_i}\cap Q, \,\text{for all $i\in \{1, \ldots, n\}$.}$$
\end{lemma}
\begin{proof}
Let $w_i\in W_{I_i}$ be arbitrary. Then 
$$w_i x_i \in \PLC(\Pi) \cap \R I_i =\PLC(I_i).$$
Thus $x_i \in U^*_{I_i}$, and Proposition~\ref{leq0} applied to $U^*_{I_i}$ yields that
$$(x_i, x_i)\leq 0.$$
On the other hand 
$$0=(x, x)=\sum_{i=1}^n (x_i, x_i),$$
and thus $(x_i, x_i)=0$, for all $i\in \{1, \ldots, n\}$. Whence
$$x_i\in U^*_{I_i}\cap Q \subseteq U^*\cap Q.$$
\end{proof}

\begin{proposition}
\label{affcomp}
Let $\eta\in U^*\cap Q$ with $\pos(\eta)\neq \emptyset$. Let $\mathbf{a}=(a_i)_{1\leq i\leq |\pos(\eta)|}$ be a transitional sequence for $\eta$ with corresponding separating index $m_{\mathbf{a}}$ and descending sequence $(\eta_i)_{1\leq i\leq |\pos(\eta)|+1}$ converging to the limit $\eta_{\mathbf{a}}$. Moreover, suppose that $\pos(\eta_{\mathbf{a}})=\emptyset$, and suppose that $I$ is a connected component of $\supp(\eta_{\mathbf{a}})$. Then $W_I =\langle r_a\mid a \in I\rangle$  is an irreducible affine standard parabolic subgroup of $W$. Furthermore, suppose that $I_1, \ldots I_k$ form a full set of connected components of $\supp(\eta_{m_{\mathbf{a}}})$ such that  
$$|\pos(\eta_{m_{\mathbf{a}}}|_{I_i})|<\infty, \quad\text{for each $i\in \{1, \ldots, k\}$},$$
where $\cdot|_{I_i}$ denotes the restriction of $\cdot$ to $\R I_i$. 
Write $x:=\sum_{i=1}^k (\eta_{m_{\mathbf{a}}}|_{I_i})$. Then 
$$x=\eta_{\mathbf{a}},$$
and, in particular, $\eta_{\mathbf{a}}\in \mathscr{K}$, and $\supp(\eta_{\mathbf{a}}) \cap \mathbf{a}_{\infty}=\emptyset$.
\end{proposition}
\begin{proof}
First, if $|\pos(\eta)|<\infty$, then by the definition of the limit of $\eta$ via a finite transitional sequence $\mathbf{a}=(a_i)_{1\leq i\leq |\pos(\eta)|}$ we have
$$\eta_{\mathbf{a}}=\eta_{m_{\mathbf{a}}}=\eta_{|\pos(\eta)|+1}\in Q\cap\mathscr{K},$$ and the claimed result holds. Thus, we may assume that $|\pos(\eta)|=\infty$.

Note that since $\eta\in Q$, it follows that $\eta_i\in Q$ for all $i\in \N$, and thus
\begin{equation}
\label{e1}
(\eta_{\mathbf{a}}, \eta_{\mathbf{a}})= (\lim_{i\to \infty}\eta_i\,,\lim_{i\to \infty}\eta_i)=\lim_{i\to\infty}(\eta_i\,, \eta_i)
                                                                                         =(\eta, \eta)=0.
\end{equation}
Also, note that the assumption of $\pos(\eta_{\mathbf{a}})=\emptyset$ implies that $(\eta_{\mathbf{a}}, a)\leq 0$ for all $a\in \Pi$, and in view of (\ref{e1}) we deduce that
\begin{equation}
\label{e2}
(\eta_{\mathbf{a}}, a)=0,\,\, \text{ for all $a\in \supp(\eta_{\mathbf{a}})$}.
\end{equation}
Now set $\eta_I:=\eta_{\mathbf{a}}|_I$. Note that if $b\in \supp(\eta_{\mathbf{a}})\setminus I$ and $a\in I\subseteq \supp(\eta_{\mathbf{a}})$, then $I$ being a connected component implies that $(b, a)=0$, and this together with (\ref{e2}) in turn yields that
\begin{equation}
(\eta_I, a) =(\eta_{\mathbf{a}}, a)=0, \,\,\text{for all $a\in I\subseteq \supp(\eta_{\mathbf{a}})$,}
\end{equation}
since $\eta_{\mathbf{a}}=\eta_I+\sum_{b\in \supp(\eta_{\mathbf{a}})\setminus I} \lambda_b b$ for some $\lambda_b \geq 0$.
Whence
\begin{equation}
\label{e3}
\eta_I \in \PLC(I)\cap \rad(\R I),
\end{equation}
and so \cite[Lemma 6.1.1]{DK94}  and \cite[Lemma 6.1.1]{DK09} yield that $W_I$ is an irreducible affine standard parabolic subgroup.
Note that then 
$$\pos(\eta_I)=\emptyset.$$ 
On the other hand, because $I$ is a connected component of $\supp(\eta_{\mathbf{a}})$, and $\supp(\eta_{\mathbf{a}})$ is a union of connected components of $\supp(\eta_{m_{\mathbf{a}}})$ by Proposition~\ref{component}, therefore, $I$ is a connected component of $\supp(\eta_{m_{\mathbf{a}}})$.
Now set $y:= \eta_{m_{\mathbf{a}}} |_I$.
Then Lemma~\ref{int} yields that $y\in U^*_I\cap Q$, and, in particular, 
$y\in \PLC(I)\cap Q$. Since $W_I$ is irreducible affine, it follows from \cite[Proposition 2.6]{HM} that 
$$\rad(\R I) = Q\cap \R I,$$
with $\dim(\rad(\R I))=1$, and therefore $y=\lambda\eta_I$ for some $\lambda >0$.
Consequently, 
$$\pos(y)=\emptyset \text{ and }\supp(y)=\supp(\eta_I)=I,$$
and hence $I=\supp(y)\subseteq \supp(x)$ by the definition of $x$ in the statement of this proposition. Thus we have established that any 
connected component of $\supp(\eta_{\mathbf{a}})$ is contained in $\supp(x)$, and therefore 
\begin{equation}
\label{e4}
\supp(\eta_{\mathbf{a}})\subseteq\supp(x).
\end{equation}
Now clearly $|\pos(x)|<\infty$, and hence Lemma~\ref{empty} yields that
\begin{equation}
\label{noainf}
\supp(x)\cap \mathbf{a}_{\infty}=\emptyset.
\end{equation}
Now Lemma~\ref{easy} and (\ref{noainf}) together imply that all the coefficients $\coeff_a(\eta_i)$ of those simple roots 
$a\in \supp(x)$ remain unchanged along the descending sequence once $i\geq m_{\mathbf{a}}$, and this means that $\eta_{\mathbf{a}}=x$.

Now having proved that $\eta_{\mathbf{a}}=x$ it is clear from (\ref{noainf}) (and also follows from combining (\ref{e4}) and (\ref{noainf})) that 
\begin{equation*}
\label{eqfin}
\supp(\eta_{\mathbf{a}}) \cap \mathbf{a}_{\infty}=\emptyset.
\end{equation*}

Finally, under the assumption of $\pos(\eta_{\mathbf{a}})=\emptyset$, we see immediately that $x=\eta_{\mathbf{a}}\in \mathscr{K}$.

\end{proof}

\begin{proposition}
  \label{zero}
  Let $\eta\in U^*\cap Q$ with $|\pos(\eta)|=\infty$. Let $\mathbf{a}=(a_i)_{i\in \N}$ be a transitional sequence for $\eta$ with $m_{\mathbf{a}}$ being a separating index, and let $\eta_{\mathbf{a}}$ be the limit of $\eta$ via $\mathbf{a}=(a_i)_{i\in \N}$.
  Suppose that $\supp(\eta_{m_{\mathbf{a}}})$ is connected and $\pos(\eta_{\mathbf{a}})=\emptyset$. Then 
  $$\eta_{\mathbf{a}}=0.$$ 
  Moreover, $\mathbf{a}_{\infty}=\supp(\eta_{m_{\mathbf{a}}})$.
\end{proposition}
\begin{proof}
Suppose for a contradiction that $\eta_{\mathbf{a}}\neq 0$. We have seen from Proposition~\ref{component} that $\supp(\eta_{\mathbf{a}})$ is a union of connected components of $\supp(\eta_{m_{\mathbf{a}}})$. Thus 
$$I:=\supp(\eta_{\mathbf{a}})=\supp(\eta_{m_{\mathbf{a}}}),$$
and Proposition~\ref{affcomp} establishes that $W_I:=\langle r_a\mid a\in I\rangle$ is an irreducible affine standard parabolic subgroup of $W$, and 
$$\eta_{\mathbf{a}}=\eta_{m_{\mathbf{a}}}.$$
But then 
$$|\pos(\eta)|\leq |\pos(\eta_{m_{\mathbf{a}}})|+m_{\mathbf{a}}=|\pos(\eta_{\mathbf{a}})|+m_{\mathbf{a}}=m_{\mathbf{a}}<\infty,$$
contradicting the assumption $|\pos(\eta)|=\infty$.

Having established that $\eta_{\mathbf{a}}=0$, it then follows readily that for each $a\in \supp(\eta_{m_{\mathbf{a}}})$, the coefficient of $a$ in the successive $\eta_i$ ($i\geq m_{\mathbf{a}}$) cannot stay the same. Now recall Lemma~\ref{easy} that for all $a\in \supp(\eta_{m_{\mathbf{a}}})\setminus \mathbf{a}_{\infty}$ we have 
$$\coeff_a(\eta_i)=\coeff_a(\eta_{m_{\mathbf{a}}}), \quad \forall i\geq m_{\mathbf{a}},$$
and consequently we have $\supp(\eta_{m_{\mathbf{a}}})\setminus \mathbf{a}_{\infty}=\emptyset$. Since $\mathbf{a}_{\infty}\subseteq \supp(\eta_i)$ for all $i\in \N$, it follows that $\mathbf{a}_{\infty}=\supp(\eta_{m_{\mathbf{a}}})$.
\end{proof}

\begin{proposition}
Let $\eta\in U^*\cap Q$ with $\pos(\eta)\neq\emptyset$. Let $\mathbf{a}=(a_i)_{1\leq i\leq |\pos(\eta)|}$ be a transitional sequence for $\eta$ with $m_{\mathbf{a}}$ being a separating index. Moreover, suppose that $\eta_{\mathbf{a}}$ is the corresponding limit with  $\pos(\eta_{\mathbf{a}})=\emptyset$. Then for all $w\in W$
$$w\eta -\eta_{\mathbf{a}}\in \PLC(\Pi)\cup\{0\}.$$
Furthermore, 
$$|\eta_{\mathbf{a}}|_1 =\inf_{w\in W} |w\eta|_1.$$
\end{proposition}
\begin{proof}
Let $J:=\supp(\eta_{m_{\mathbf{a}}})$. Let $I_1, \ldots, I_k$ be a full set of connected components of $J$ with 
$$|\pos(\eta_{m_{\mathbf{a}}}|_{I_i})|<\infty, \quad \forall i\in\{1, \ldots, k\},$$ 
and let $M:=\uplus_{i=1}^{k}I_i$.
Then Proposition~\ref{affcomp} yields that 
$$\eta_{\mathbf{a}}=\eta_{m_{\mathbf{a}}}|_M,$$
where $\cdot|_A$ again means the restriction of $\cdot$ to the subspace spanned by $A$.

Now set $x':=\eta_{m_{\mathbf{a}}}-\eta_{\mathbf{a}}=\eta_{m_{\mathbf{a}}}|_{J\setminus M}$.

Observe that Proposition~\ref{component} yields that $J\setminus M$ is a union of connected components when $|\pos(\eta)|=\infty$; whereas when $|\pos(\eta)|<\infty$ then it is readily seen that $J\setminus M =\emptyset$. Also observe that $\eta_{m_{\mathbf{a}}}\in U^*\cap Q$, since $\eta\in U^*\cap Q$. These two observations together with Lemma~\ref{int} yield that 
$$x'\in U^*_{J\setminus M}\cap Q \subseteq U^*\cap Q.$$ 
Thus $\forall a\in M$ and $\forall w\in W$
\begin{equation}
\label{>0}
\coeff_a(w x')\geq 0.
\end{equation}
On the other hand $\pos(\eta_{\mathbf{a}})=\emptyset$, and hence Lemma~\ref{pos&conn} yields that $\forall a\in \Pi$ and $\forall w\in W$
$$\coeff_a(w \eta_{\mathbf{a}})\geq \coeff_a(\eta_{\mathbf{a}}).$$
Thus, in particular, $\forall a\in M=\supp(\eta_{\mathbf{a}})$ and $\forall w\in W$
\begin{equation*}
  \coeff_a(w \eta_{m_{\mathbf{a}}})-\underbrace{\coeff_a(w x')}_\text{$\geq 0$ by (\ref{>0})}=\coeff_a(w \eta_{\mathbf{a}})\geq \coeff_a(\eta_{\mathbf{a}})
\end{equation*}

That is, $\forall a\in M$ and $\forall w\in W$
\begin{equation}
\label{ee1}
  \coeff_a(w \eta_{m_{\mathbf{a}}})\geq \coeff_a(\eta_{\mathbf{a}}).
\end{equation}
While since $\eta_{m_{\mathbf{a}}}\in U^*$, it follows that $\forall a\in \Pi\setminus M$ and $w\in W$
\begin{equation}
\label{ee2}
\coeff_a(w\eta_{m_{\mathbf{a}}})\geq 0=\coeff_a(\eta_{\mathbf{a}}).  
\end{equation}
Combining (\ref{ee1}) and (\ref{ee2}), we deduce that $\forall a\in \Pi$ and $\forall w\in W$
$$\coeff_a(w\eta_{m_{\mathbf{a}}}) \geq \coeff_a(\eta_{\mathbf{a}}).$$
But $\eta_{m_{\mathbf{a}}} =r_{a_{m_{\mathbf{a}}-1}}\cdots r_{a_1}\eta$, hence $\forall a\in\Pi$ and $\forall w\in W$
$$w\eta - \eta_{\mathbf{a}}\in \PLC(\Pi)\cup \{0\}.$$
In particular, 
$$|\eta_{\mathbf{a}}|_1\leq \inf_{w\in W} |w\eta|_1.$$
On the other hand, it follows from the definition of $\eta_{\mathbf{a}}$ that 
$$|\eta_{\mathbf{a}}|_1\geq \inf_{w\in W} |w\eta|_1,$$ 
whence equality.
\end{proof}

\begin{corollary}
\label{equal}
Suppose that $\eta\in U^*\cap Q$ with $\pos(\eta)\neq \emptyset$. Suppose that $\mathbf{a}=(a_i)_{1\leq i\leq |\pos(\eta)|}$ is a transitional sequence for $\eta$ with $m_{\mathbf{a}}$ being a separating index. Moreover, suppose that $\eta_{\mathbf{a}}$ is the corresponding limit with  $\pos(\eta_{\mathbf{a}})=\emptyset$.
Suppose that $\mathbf{b}=(b_i)_{1\leq i\leq |\pos(\eta)|}\in T(\eta)$ is yet another transitional sequence for $\eta$, giving rise to the limit $\eta_{\mathbf{b}}$. Then 
$$\eta_{\mathbf{b}}-\eta_{\mathbf{a}}\in \PLC(\Pi)\cup\{0\}.$$ 
Furthermore, if $\pos(\eta_{\mathbf{b}})=\emptyset$, then 
$$\eta_{\mathbf{a}}=\eta_{\mathbf{b}}.$$ 
\end{corollary}
\begin{proof}
The last proposition affirms that $\forall a\in \Pi$ and $\forall w\in W$
$$\coeff_a(w\eta)\geq \coeff_a(\eta_{\mathbf{a}}).$$
Thus, in particular, $\forall a\in \Pi$ and $\forall n$ with  $1\leq n\leq |\pos(\eta)|$
$$\coeff_a(r_{b_{n}}\ldots r_{b_1}\eta)\geq \coeff_a(\eta_{\mathbf{a}}).$$
Consequently, $\forall a\in \Pi$
$$\coeff_a(\eta_{\mathbf{b}})\geq \coeff_a(\eta_{\mathbf{a}}),$$
whence $$\eta_{\mathbf{b}}-\eta_{\mathbf{a}}\in \PLC(\Pi)\cup\{0\}.$$

Next, if $\pos(\eta_{\mathbf{b}})=\emptyset$, then $\eta_{\mathbf{b}}$ satisfies the same requirements as $\eta_{\mathbf{a}}$, and hence a symmetric argument can be applied with the roles of $\eta_{\mathbf{a}}$ and $\eta_{\mathbf{b}}$ interchanged to establish that 
$$\eta_{\mathbf{a}}-\eta_{\mathbf{b}}\in \PLC(\Pi)\cup\{0\},$$
whence $$\eta_{\mathbf{a}}=\eta_{\mathbf{b}}.$$

\end{proof}

\begin{lemma}
  \label{large}
  Let $\eta\in U^*$, and let $(w_i)_{i\in \N}$ be an infinite sequence of elements in $W$ such that 
  $$\eta_{\infty}:=\lim_{i\to \infty} w_i\eta \quad\text{ exists.}$$
  Let $x\in \pos(\eta_{\infty})$. If $w_i^{-1}x\in \Phi^+$ for all $i\in \N$, then the set
  $$\{\, w_i^{-1}x\mid i\in \N \,\}$$
  is finite, and
  $$(w_i\eta, x)=(\eta_{\infty}, x)$$
  holds for all sufficiently large $i$.
\end{lemma}
\begin{proof}
Note that for $x\in \pos(\eta_{\infty})$
$$\lim_{i\to\infty} (\eta, w_i^{-1}x)=\lim_{i\to\infty}(w_i\eta, x)=(\eta_{\infty}, x)>0.$$
Thus there exists some $n\in \N$ such that 
$$(\eta, w_i^{-1}x)>\frac{(\eta_{\infty}, x)}{2}, \quad\forall i\geq n.$$
Recall that it follows from Proposition~\ref{epsilon} 
$$\bigg|\bigg\{y\in \Phi^+\mid (y, \eta)\geq \frac{(\eta_{\infty}, x)}{2}\bigg\}\bigg|<\infty.$$
Since $\{w_i^{-1}x\mid i\geq n\}\subseteq \{y\in \Phi^+\mid (y, \eta)\geq \frac{(\eta_{\infty}, x)}{2}\}$, it follows that the set
$$\{w_i^{-1}x\mid i\geq n\}$$
is finite, whence 
$$\{w_i^{-1}x\mid i\in \N\}$$
is also finite. Consequently, 
$$(\eta, w_i^{-1}x)=(w_i\eta, x)=(\eta_{\infty}, x),$$
once $i$ is large enough.
\end{proof}

\begin{proposition}
\label{pos_cap}

Let $\eta\in U^*$ with $|\pos(\eta)|=\infty$. Let $\mathbf{a}=(a_i)_{i\in \N}$ be a transitional sequence for $\eta$. Let $\eta_{\mathbf{a}}$ be the corresponding limit and let $m_{\mathbf{a}}$ be a corresponding separating index. 
Then $\forall n\geq m_{\mathbf{a}}$
$$\pos(\eta_{\mathbf{a}})=\pos(\eta_n)\cap \{\,x\in \Phi^+\mid (x, a_i)=0\,\,\forall i\geq m_{\mathbf{a}} \,\},$$
and, $\forall x\in \pos(\eta_{\mathbf{a}})$
$$(\eta_n, x)=(\eta_{\mathbf{a}}, x).$$
\end{proposition}
\begin{proof}
First note that if $x\in \pos(\eta_n)\cap \{\,x\in \Phi^+\mid (x, a_i)=0\,\,\forall i\geq m_{\mathbf{a}} \,\}$, then 
$$(\eta_i, x)=(\eta_n, x)>0, \quad \forall i\geq n.$$
In particular, for all those $x$,
$$(\eta_{\mathbf{a}}, x)=\lim_{i\to \infty}(\eta_i, x)=(\eta_n, x)>0.$$
Hence 
$$\pos(\eta_n)\cap \{\,x\in \Phi^+\mid (x, a_i)=0\,\,\forall i\geq m_{\mathbf{a}} \,\}\subseteq\pos(\eta_{\mathbf{a}}).$$

Conversely, let $x\in \pos(\eta_{\mathbf{a}})$. Now for each $i>n\geq m_{\mathbf{a}}$, define
$$w_i:=r_{a_{i-1}}\cdots r_{a_n}.$$
Then for all $i>n\geq m_{\mathbf{a}}$
$$w_i\eta_{n}=\eta_i.$$
Recall Lemma~\ref{form0} and Remark~\ref{eta_a}, and we see that
$$(a, \eta_{\mathbf{a}})=0,\quad \forall a\in \mathbf{a}_{\infty}.$$
Thus the requirement $x\in \pos(\eta_{\mathbf{a}})$ means that 
$\supp(x)$ is not contained in $\mathbf{a}_{\infty}$, and, in particular, 
$\supp(x)$ is not contained in $\{\,a_i\mid i> n\,\}$. Thus 
$$w_i^{-1}x\in \Phi^+, \quad\forall i>n\geq m_{\mathbf{a}}.$$
Furthermore, since $\lim_{i\to\infty}(w_i\eta_n)=\lim_{i\to \infty}\eta_i=\eta_{\mathbf{a}}$ exists, it follows from Lemma~\ref{large} that
$$(\eta_i, x)=(w_i\eta_n, x)=(\eta_n, w_i^{-1}x)=(\eta_{\mathbf{a}}, x)$$
for all $i$ sufficiently large.
This implies that 
$$(x, a_i)=0, \text{ for all $i$ sufficiently large}.$$
But once $i$ is sufficiently large, $a_i\in \mathbf{a}_{\infty}$, this implies that 
$$(x, a_i) = 0, \quad \forall i\geq m_{\mathbf{a}}.$$
Then 
$$(\eta_i, x)=(\eta_n, x), \quad \forall i\geq m_{\mathbf{a}},$$
and so 
$$(\eta_n, x)=(\eta_{\mathbf{a}}, x)>0,$$
and 
$$x\in \pos(\eta_n)\cap \{\,x\in \Phi^+\mid (x, a_i)=0\,\,\forall i\geq m_{\mathbf{a}} \,\}.$$ 

\end{proof}

We close this section with a stronger version of Proposition~\ref{zero}:
\begin{theorem}
  \label{zero2}
  Let $\eta\in U^*\cap Q$ with $|\pos(\eta)|=\infty$. Let $\mathbf{a}=(a_i)_{i\in \N}$ be a transitional sequence for $\eta$, with $m_{\mathbf{a}}$ being a separating index for $\mathbf{a}=(a_i)_{i\in \N}$ and $\eta_{\mathbf{a}}$ being the corresponding limit of $\eta$ via $\mathbf{a}=(a_i)_{i\in \N}$.
  Suppose further that $\supp(\eta_{m_{\mathbf{a}}})$ is connected. Then
  $$\eta_{\mathbf{a}}=0 \quad\text{ and  }\quad \mathbf{a}_{\infty}=\supp(\eta_{m_{\mathbf{a}}}).$$
\end{theorem}

\begin{proof}
Suppose for a contradiction that the limit $\eta_{\mathbf{a}}\neq 0$. Then it follows from Proposition~\ref{zero} that $\pos(\eta_{\mathbf{a}})\neq \emptyset$. Recall that we have established in Proposition~\ref{component} that $\supp(\eta_{\mathbf{a}})$ is a union of connected components of $\supp(\eta_{m_{\mathbf{a}}})$, and hence under the present assumption we have that 
\begin{equation}\label{I}
  I:=\supp(\eta_{m_{\mathbf{a}}})=\supp(\eta_{\mathbf{a}}),
\end{equation}
with $I$ being connected.


Let $\mathbf{b}=(b_j)_{1\leq j\leq |\pos(\eta_{\mathbf{a}})|}\in T(\eta_{\mathbf{a}})$, that is, $\mathbf{b}=(b_j)_{1\leq j\leq |\pos(\eta_{\mathbf{a}})|}$ is a transitional sequence for $\eta_{\mathbf{a}}$. Next for each $1\leq j\leq |\pos(\eta_{\mathbf{a}})|$ let us define 
\begin{align*}
\eta'_1&:=\eta_{\mathbf{a}},\\
\noalign{\hbox{and whenever $j\leq |\pos(\eta_{\mathbf{a}})|$}}
\eta'_{j+1}&:=r_{b_j} \eta'_j,
\end{align*}
where $b_j\in \Pi\cap \pos(\eta'_j)$ with the further requirement that 
$$(b_j, \eta'_j)\geq (a, \eta'_j), \quad \forall a\in \Pi\cap\pos(\eta'_j).$$
Then Remark~\ref{ts} and Lemma~\ref{eg2} ensure that for the corresponding limit ${\eta_{\mathbf{a}}}_{\mathbf{b}}$ of $\eta_{\mathbf{a}}$ via $\mathbf{b}=(b_j)_{1\leq j\leq |\pos(\eta_{\mathbf{a}})|}$,  
$$\pos({\eta_{\mathbf{a}}}_{\mathbf{b}})=\emptyset.$$
Let $m=m_{\mathbf{b}}$ be a separating index for $\mathbf{b}=(b_j)_{1\leq j\leq |\pos(\eta_{\mathbf{a}})|}$, and let 
$$K:=K_1\uplus \cdots \uplus K_l,$$
where $K_1, \ldots, K_l$ form a full set of pairwise disjoint connected components of $\supp(\eta'_m)$ where $\eta'_m =r_{b_{m-1}}\cdots r_{b_1}\eta_{\mathbf{a}}$, such that 
$$|\pos(r_{b_{m-1}}\cdots r_{b_1}\eta_{\mathbf{a}}|_{K_k})|<\infty, \quad \text{for each $k\in \{1,\ldots, l\}$}.$$
Next set
\begin{equation}\label{ydef}
y:=r_{b_{m-1}}\cdots r_{b_1} (\eta_{\mathbf{a}})|_{K}=\eta'_m|_K.
\end{equation}
Note that under this construction, Proposition~\ref{affcomp} asserts that 
\begin{equation}\label{limit limit}
  y= {\eta_{\mathbf{a}}}_{\mathbf{b}}.
\end{equation}
 
Next, for all $1\leq j\leq |\pos(\eta_{\mathbf{a}})|$, construct
$$b_{j,1}:=r_{b_1}\cdots r_{b_{j-1}}b_j.$$
Then Lemma~\ref{pos} ensures that 
$$b_{j, 1} \in \pos(\eta_{\mathbf{a}}), \quad 1\leq j\leq |\pos(\eta_{\mathbf{a}})|,$$
moreover, recall Proposition~\ref{pos_cap} that 
$$\pos(\eta_{\mathbf{a}})=\pos(\eta_{m_{\mathbf{a}}})\cap\{x\in \Phi^+\mid (x, a_i)=0,\,\,\forall i\geq m_{\mathbf{a}}\},$$
and hence for each $1\leq j\leq |\pos(\eta_{\mathbf{a}})|$
$$(b_{j, 1}, a_i)=(r_{b_1}\cdots r_{b_{j-1}}b_j, a_i)=0,\quad \forall i\geq m_{\mathbf{a}}.$$
In particular, for all $i\geq m_{\mathbf{a}}$
\begin{align*}
(a_i, b_1)&=0,\\
\noalign{\hbox{and }}
(a_i, r_{b_1}b_2)&=(a_i,b_2)-2(b_1, b_2)(a_i, b_1)=0, \\
\end{align*}
that is, $(a_i, b_2)=0$ and so on. Hence for each $1\leq j\leq |\pos(\eta_{\mathbf{a}})|$
\begin{equation}
\label{aibj0}
(a_i, b_j)=0,\quad \forall i\geq m_{\mathbf{a}}.
\end{equation}
Now Lemma~\ref{easy} implies that
\begin{equation}\label{ainf}
\supp(\eta_{m_{\mathbf{a}}}-\eta_{\mathbf{a}})\subseteq \mathbf{a}_{\infty}=\{a_i\mid i\geq m_{\mathbf{a}}\},  
\end{equation}
and hence we deduce that 
\begin{align*}
  &r_{b_{m-1}}\cdots r_{b_1} \eta_{m_{\mathbf{a}}}  = r_{b_{m-1}}\cdots r_{b_1} (\eta_{m_{\mathbf{a}}}-\eta_{\mathbf{a}}) +r_{b_{m-1}}\cdots r_{b_1} (\eta_{\mathbf{a}}) \\
  & =  \underbrace{(\eta_{m_{\mathbf{a}}}-\eta_{\mathbf{a}})}_\text{$(\eta_{m_{\mathbf{a}}}-\eta_{\mathbf{a}})$ fixed by $r_{b_{m-1}}\cdots r_{b_1}$, following (\ref{aibj0}) and (\ref{ainf}) }+r_{b_{m-1}}\cdots r_{b_1} (\eta_{\mathbf{a}}).
\end{align*}
Note that $\eta_{m_{\mathbf{a}}}-\eta_{\mathbf{a}}\in \PLC(\Pi)$ and $r_{b_{m-1}}\cdots r_{b_1} (\eta_{\mathbf{a}})\in U^*\subset\PLC(\Pi)\cup\{0\}$. It then follows from the above that 
\begin{equation}
\label{suppcontain}
 \supp(r_{b_{m-1}}\cdots r_{b_1}\eta_{\mathbf{a}})\subseteq \supp(r_{b_{m-1}}\cdots r_{b_1} \eta_{m_{\mathbf{a}}}). 
\end{equation} 
Furthermore,
\begin{equation}\label{inainf}
\supp(r_{b_{m-1}}\cdots r_{b_1} \eta_{m_{\mathbf{a}}})\setminus\supp(r_{b_{m-1}}\cdots r_{b_1}\eta_{\mathbf{a}})
\subseteq \supp(\eta_{m_{\mathbf{a}}}-\eta_{\mathbf{a}})\subseteq \mathbf{a}_{\infty}.
\end{equation}
On the other hand
\begin{equation}\label{assumption}
\underbrace{\mathbf{a}_{\infty}\subseteq \supp(\eta_{m_{\mathbf{a}}})}_\text{by Lemma~\ref{trivial}} \underbrace{=\supp(\eta_{\mathbf{a}})}_\text{by (\ref{I})}.
\end{equation}
Now (\ref{assumption}) and (\ref{aibj0}) together imply that
\begin{equation}\label{ainfin}
\mathbf{a}_{\infty}\subseteq \supp(r_{b_{m-1}}\cdots r_{b_1} \eta_{m_{\mathbf{a}}})\cap \supp(r_{b_{m-1}}\cdots r_{b_1} \eta_{\mathbf{a}}).
\end{equation}
Upon combining (\ref{inainf}) and (\ref{ainfin}) we see that
\begin{align*}
 \supp(r_{b_{m-1}}\cdots r_{b_1} \eta_{m_{\mathbf{a}}})&\setminus\supp(r_{b_{m-1}}\cdots r_{b_1}\eta_{\mathbf{a}})\subseteq \mathbf{a}_{\infty}\\  &\subseteq \supp(r_{b_{m-1}}\cdots r_{b_1} \eta_{m_{\mathbf{a}}})\cap \supp(r_{b_{m-1}}\cdots r_{b_1} \eta_{\mathbf{a}}).
\end{align*}
Consequently, in view of (\ref{suppcontain}) we deduce from the above that
\begin{equation}\label{ainfeq}
\supp(r_{b_{m-1}}\cdots r_{b_1} \eta_{m_{\mathbf{a}}})=\supp(r_{b_{m-1}}\cdots r_{b_1}\eta_{\mathbf{a}}).
\end{equation} 
Now recall the definition of $y$ given in (\ref{ydef}), and let $J$ be a connected component of $\supp(y)=K$. 
By (\ref{ainfeq})
$$J\subseteq K\subseteq \supp(r_{b_{m-1}}\cdots r_{b_1}\eta_{\mathbf{a}} )=\supp(r_{b_{m-1}}\cdots r_{b_1}\eta_{m_{\mathbf{a}}}),$$ and furthermore,
$W_J:=\langle r_a\mid a\in J\rangle$ is an irreducible affine standard parabolic subgroup of $W$. Now since 
$r_{b_{m-1}}\cdots r_{b_1} \eta_{m_{\mathbf{a}}}\in U^*\cap Q$, it follows from Lemma~\ref{int} that 
\begin{equation}\label{iso}
r_{b_{m-1}}\cdots r_{b_1} \eta_{m_{\mathbf{a}}}|_J\in Q\cap \R J.
\end{equation}
Since $W_J$ is an irreducible affine standard parabolic subgroup, it follows that the isotropic set of the subspace spanned by $J$ (that is, $Q\cap\R J$) coincides with the radical of the subspace spanned by $J$. Therefore 
$$(a, r_{b_{m-1}}\cdots r_{b_1} \eta_{m_{\mathbf{a}}}|_J )=0,\quad \forall a\in J.$$
Since $\supp(y)=K$ is the union of some connected components of $\supp(r_{b_{m-1}}\cdots r_{b_1} \eta_{m_{\mathbf{a}}})$,
we deduce that for all $a\in \supp(y)=K$,
\begin{equation}\label{supp y}
(r_{b_{m-1}}\cdots r_{b_1}\eta_{m_{\mathbf{a}}}, a)=0.
\end{equation}
On the other hand we point out that one consequence of (\ref{aibj0}) and (\ref{assumption}) is that 
\begin{equation}\label{suppeq}
\mathbf{a}_{\infty}\subseteq \supp({\eta_{\mathbf{a}}}_{\mathbf{b}})=\supp(y)=K.  
\end{equation}
Therefore, for all $i\geq m_{\mathbf{a}}$ (that is, for all $a_i\in \mathbf{a}_{\infty}$)
$$(\eta_{m_{\mathbf{a}}}, a_i)=\underbrace{(r_{b_{m-1}}\cdots r_{b_1}\eta_{m_{\mathbf{a}}}, a_i)=0}_\text{by (\ref{supp y}) and (\ref{suppeq})},$$
which is a contradiction to the definition of the $a_i$'s. Whence $\eta_{\mathbf{a}}=0$. 

Finally, it follows from Proposition~\ref{zero} that 
$$\mathbf{a}_{\infty}=\supp(\eta_{m_{\mathbf{a}}})=\supp(\eta_m), \quad \forall m\geq m_{\mathbf{a}}.$$

\end{proof}

%
%
%

\section{Accumulations of infinite dominance chains}
\label{sec:dominance-accumulation}

In this section, we study the accumulation points of infinite normalised sequence $(\widehat{x_i})_{i\in \N}$ where $(x_i)_{i\in \N}$ is an infinite injective sequence of positive roots totally ordered by dominance, that is, 
$$x_{i+1}\dom x_i \text{ and } x_{i+1}\neq x_i, \quad\forall i\in \N.$$ 
The following was asked in \cite[Section 7.7]{DHR13}: whether or not the normalised sequence $(\widehat{x_i})_{i\in \N}$ is convergent. 
In a forthcoming paper, we shall resolve this question in full. 

It is still profitable, in our view, to study how the techniques that we have developed in the previous sections of this paper may be applied to study infinite dominance chains, and, indeed, the dominance partial ordering itself further. We first show that the uniqueness of accumulation holds in all infinite rank-$3$ Coxeter groups. Next, we show that if such a normalised sequence possesses more than one accumulation point, then any two of these distinct accumulation points, say,  $\eta_1$ and $\eta_2$, would enjoy a number of similarities, including the following:
\begin{itemize}
\item[(i)] $|\pos(\eta_1)|<\infty \implies \eta_1=\eta_2$; 
\item[(ii)] $\supp(\eta_1)=\supp(\eta_2)=\supp(x_i)$, once $i$ is sufficiently large, and, in particular, this support is connected;
\item[(iii)] $\big\{ x\in \Phi^+\mid (x, \eta_1)\diamond 0 \big\} = \big\{ x\in \Phi^+\mid (x, \eta_2)\diamond 0 \big\}$ where $\diamond$ can assume each of the binary relations $<$, $=$, or $>$, respectively;
\item[(iv)] $\big\{ x\in U^* \mid (x, \eta_1)\diamond 0 \big\} = \big\{ x\in U^* \mid (x, \eta_2)\diamond 0 \big\}$ where $\diamond$ can assume each of the binary relations $<$, $=$, or $>$, respectively;
\item[(v)] In the case where both 
   $$\limsup_{i_j\to\infty}(\eta_2, x_{i_j})|x_{i_j}|_1 >0 \text{ and } \limsup_{i_k\to\infty}(\eta_1, x_{i_k})|x_{i_k}|_1 >0$$ 
   then  $\eta_1=\eta_2$, where $\eta_1$ is the limit of the subsequence $(\widehat{x_{i_j}})$ and $\eta_2$ is the limit of the subsequence $(\widehat{x_{i_k}})$.     
\end{itemize}
We hope that these findings may prove to be useful in further investigations of these accumulation points.

We also give a number of conditions under which uniqueness of accumulation can be expected. 
\begin{lemma}
\label{limorth}
Suppose that $(x_i)_{i\in \N}\subseteq \Phi^+$ is an infinite injective sequence of positive roots and suppose that $\eta_1$ and $\eta_2$ are two accumulation points for the normalised sequence $(\widehat{x_i})_{i\in \N}$ such that 
$$(x_i, \eta_1)\geq 0 \text{ and } (x_i, \eta_2)\geq 0, \quad \forall i\in \N.$$
Then $$(\eta_1, \eta_2)=0.$$ 
\end{lemma}
\begin{proof}
Suppose that $(x_{i_{k}})_{k\in \N}$ is a subsequence of $(x_i)_{i\in \N}$ with 
$$\lim_{k\to\infty} \widehat{x_{i_k}}=\eta_2.$$
Then 
$$(\eta_1, \eta_2)=\lim_{k\to\infty} (\eta_1, \widehat{x_{i_{k}}})\geq 0.$$
On the other hand, since $\eta_1, \eta_2\in U^*$  it follows that $(\eta_1, \eta_2)\leq 0$. Therefore $(\eta_1, \eta_2)=0$. 
\end{proof}

Note that Lemma~\ref{limorth} applies to infinite injective sequence of positive roots totally ordered by dominance, and indeed, as a special case of Lemma~\ref{limorth} we recover the following result first proven in \cite[Proposition 6.8]{DHR13}:

\begin{lemma}
\label{specialcase}
Suppose that $(x_i)_{i\in \N}$ is an infinite injective sequence of positive roots totally ordered by dominance, that is, 
$$x_{i+1}\dom x_i \text{ and } x_{i+1}\neq x_i, \quad\forall i\in \N.$$
Suppose that $\eta_1$ and $\eta_2$ are two accumulation points for the corresponding normalised sequence $(\widehat{x_i})_{i\in \N}$. Then 
$$(\eta_1, \eta_2)=0.$$
\qed
\end{lemma}
\begin{remark}
  \label{isotropicline}
Note that in the situation described in Lemma~\ref{specialcase} the straight line passing through $\eta_1$ and $\eta_2$ is entirely isotropic. Indeed, an arbitrary point $\zeta$ of this line has the form 
$$\zeta =\lambda \eta_1 +(1-\lambda)\eta_2, \quad \lambda\in \R.$$
Since $(\eta_1, \eta_1) =(\eta_1, \eta_2)=(\eta_2, \eta_2)=0$, we have that 
$$(\zeta, \zeta)=\lambda^2(\eta_1, \eta_1)+2\lambda (1-\lambda)(\eta_1, \eta_2)+(1-\lambda)^2(\eta_2, \eta_2) =0.$$
This observation combined with the following standard characterisation of the normalised isotropic sets for rank-$3$ infinite Coxeter groups rules out the
existence of such straight lines in the normalised isotropic sets of rank-$3$ infinite Coxeter groups, and hence establishing unique accumulation for these groups. The following characterisation of normalised isotropic sets is well known to the experts of the field, but we choose to include a proof here for completeness. We note that it has also been observed in \cite[Proposition 2.8]{HMN} that the normalised isotropic sets of weakly hyperbolic Coxeter groups do not contain two distinct points that are orthogonal to each other with respect to the bilinear form $(\cdot ,\cdot)$, and therefore, for weakly hyperbolic Coxeter groups, the isotropic sets do not contain straight lines. 
   
\end{remark}

\begin{lemma}
\label{rank3}
Suppose that $W$ is a rank-$3$ infinite Coxeter group.  
Then the normalised isotropic set
$$Q\cap V_1 =\{v\in V\mid \text{$(v, v)=0$ and $|v|_1 =1$}\}$$
does not contain a straight line.
\end{lemma}
\begin{proof}
Denote the set of simple roots for $W$ to be $\Pi=\{a, b, c\}$. Suppose, for a contradiction, that $Q\cap V_1$ contains a straight line $L$.
Then $L$ must intersect at least two of $L(a, b)$, $L(a, c)$, and $L(b, c)$, where
$$L(\alpha, \beta):=\{\lambda \alpha+(1-\lambda)\beta\mid \lambda\in \R\}$$
is the straight line passing through $\alpha$ and $\beta$.
By symmetry, performing a permutation of the simple roots if necessary, we may assume that $L$ intersects $L(a, b)$ and $L(a, c)$. 
Since $L$ intersects $L(a, b)$, we may assume that 
$$\widehat{a+kb}\in L, \quad \exists k\in \R.$$
Now because $L\subseteq Q$, we see that 
$$0=(a+kb, a+kb)=1+2k(a, b)+k^2.$$
Hence 
$$k(a, b)=-\frac{1+k^2}{2},$$
and, in particular, $k>0$ for $(a, b)\leq 0$.
Similarly, since $L$ intersects $L(a, c)$, we may assume that
$$\widehat{a+lc}\in L, \quad \exists l\in \R.$$
Again, because $L\subseteq Q$, we see that 
$$0=(a+lc, a+lc)=1+2l(a, c)+l^2.$$
Hence 
$$l(a, c)=-\frac{1+l^2}{2},$$
and, in particular, $l>0$ for $(a, c)\leq 0$.

Next, note that if $x, y\in L$ then there exists some $\mu\in \R$ with $\mu(x+y)\in L$ (indeed, consider the parallelogram with vertices the origin, $x$, $x+y$, and $y$, the diagonal passing through the origin and the vertex $x+y$ intersects $L$). Thus 
$$0=\mu^2(x+y, x+y).$$
Consequently, 
$$0=(x+y, x+y)=(x, x)+2(x, y)+(y,y)=2(x, y),$$
since $x, y$ are isotropic.  

Now combine the above two observations, we have
\begin{align*}
  0&=(a+kb, a+lc)=(a, a)+l(a, c) + k(a, b) +kl (b,c)  \\
   &\leq 1+l(a, c)+k(a,b)\quad \text{ $\because k, l >0$ and $(b, c)\leq 0$} \\
   &=1-\frac{1+l^2}{2}-\frac{1+k^2}{2}\\
   &=-\frac{l^2+k^2}{2} <0,
\end{align*}
a contradiction.  
\end{proof}

Equipped with the above characterisation of the normalised isotropic sets of rank-$3$ Coxeter groups, and in view of Remark~\ref{isotropicline}, we immediately have the following:
\begin{proposition}
Suppose that $W$ is a rank-$3$ infinite Coxeter group, and suppose that $(x_i)_{i\in \N}$ is an infinite injective sequence of positive roots totally ordered by dominance, that is, 
$$x_{i+1}\dom x_i \text{ and } x_{i+1}\neq x_i, \quad\forall i\in \N.$$
Then the corresponding normalised sequence $(\widehat{x_i})_{i\in \N}$ is convergent with a unique accumulation point.
\qed 
\end{proposition}

\begin{proposition}
\label{suppcon}
Let $(x_i)_{i\in \N}\subset \Phi^+$ be an infinite injective sequence of positive roots such that $\lim_{i\to\infty}\widehat{x_i} =\eta$, 
and $(x_i, \eta)\geq 0$ for all $i\in \N$.  Then there exists some $N\in \N$ with
$$\supp(x_i) =\supp(\eta)$$
for all $i\geq N$. In particular, $\supp(\eta)$ is connected, and the sequence of supports $(\supp(x_i))_{i\in \N}$ stabilises.   
\end{proposition}

\begin{proof}
Note that $\lim_{i\to\infty}\widehat{x_i}=\eta$ implies that $\supp(\eta)\subseteq \supp(x_i)$ for $i$ large enough. (Indeed, it is possible that upon normalisation, some simple roots appearing in the expression of $x_i$ may have their contribution overwhelmed by the sheer magnitude of $|x_i|_1$).

It remains to show that $\exists N\in \N$ with $\supp(x_i)\subseteq \supp(\eta)$, for all $i\geq N$.
There are two situations that we need to consider separately, the first being that $\pos(\eta)=\emptyset$ and the second being that 
$\pos(\eta)\neq \emptyset$.  

First, suppose that $\pos(\eta)=\emptyset$. Then the assumption $(x_i, \eta)\geq 0$ for all $i\in \N$ implies that
\begin{equation}
\label{one}
\text{$(a, \eta)=0$ for all $a\in \supp(x_i)$ and all $i\in \N$.}
\end{equation}
Now by the opening remark, we may take $N\in \N$ large enough such that 
$$\supp(\eta)\subseteq\supp(x_i), \quad \forall i\geq N.$$
Fix one such $i$, and suppose that $\supp(x_i)\setminus\supp(\eta)\neq \emptyset$. Since $\supp(x_i)$ is connected ($x_i$ being a root), there exists some $b\in \supp(x_i)\setminus\supp(\eta)$ with $(b,c)<0$ for some $c\in \supp(\eta)$.
Write $\eta=\sum_{d\in \supp(\eta)}\coeff_d(\eta)d$, and we may conclude that all those $\coeff_d(\eta)>0$, since $\eta\in U^*$.
Then 
\begin{align*}
  0=(b, \eta)&=\coeff_c(\eta)(b, c)+\sum_{d\in \supp(\eta)\setminus\{c\}}\coeff_d(\eta)(b, d)\\
  &\leq \coeff_c(\eta)(b, c)<0,
\end{align*}
which is a contradiction, whence $\supp(x_i)=\supp(\eta)$ in the case when $\pos(\eta)=\emptyset$.

Next, suppose that $\pos(\eta)\neq \emptyset$. Set $W':=\langle r_a\mid a\in \supp(\eta)\rangle$, and choose $w_0\in W'$ such that
$$|\supp(w_0\eta)|=\min_{w\in W'}\{\,|\supp(w\eta) |\,\}.$$
Write $w_0=r_{\alpha_1}\cdots r_{\alpha_m}$ where $\alpha_1, \ldots, \alpha_m\in \supp(\eta)$, and define the corresponding
$$I:=\{\alpha_1, \ldots, \alpha_m\}\subseteq \supp(\eta),$$
and 
$$W'':=\langle\, r_a \mid a\in \supp(w_0 \eta)\,\rangle \subseteq W'.$$
Note that the definition of $W''$ has the consequence that 
$$\supp(ww_0\eta)\subseteq\supp(w_0\eta), \quad \forall w\in W''.$$ 
This together with the minimality of $|\supp(w_0\eta)|$ also implies that 
$$\supp(w_0\eta)\subseteq\supp(ww_0\eta),\quad  \forall w\in W''\subseteq W'.$$  
Consequently,
\begin{equation}
\label{conn0}
\supp(ww_0\eta)=\supp(w_0\eta), \quad \forall w\in W''.  
\end{equation} 
Now if $\pos(w_0\eta)=\emptyset$ then as before (with $w_0\eta$ in place of $\eta$ and $w_0x_i$ in place of $x_i$) we conclude that there exists some $N_1\in \N$ with
\begin{equation}\label{conn}
\supp(w_0 \eta)=\supp(w_0x_i), \quad \forall i\geq N_1.  
\end{equation}
Else, $\pos(w_0\eta)\neq\emptyset$, and there exists $a_1\in \supp(w_0\eta)$ with $(w_0\eta, a_1)>0$. Since $\lim_{i\to \infty}\widehat{x_i}=\eta$  and $(\,,\,)$ is continuous, it follows that there exists some $N_2\in \N$ with
$$(w_0 x_i, a_1) >0, \quad\forall i\geq N_2,$$
moreover, the finiteness of $N(w_0)$ and the injectivity of the $(x_i)_{i\in \N}$ imply that by enlarging $N_2$ if necessary we may also ensure  $w_0x_i\in \Phi^+$, for all $i\geq N_2$. Now fix such an $i$, then 
$\dep(r_{a_1}w_0x_i)=\dep(w_0 x_i)-1$. Note that by (\ref{conn0}) we have
\begin{equation}
\label{conn1}
  a_1\in \supp(w_0\eta)=\supp(r_{a_1}w_0\eta),
\end{equation}
and also note that 
\begin{equation}
\label{conn2}
(r_{a_1}w_0x_i, r_{a_1}w_0\eta)=(x_i, \eta)\geq 0,
\end{equation}
while
\begin{equation}\label{conn3}
 (r_{a_1} w_0 x_i, a_1) =-(w_0x_i, a_1 )<0. 
\end{equation}
Now (\ref{conn1}), (\ref{conn2}) and (\ref{conn3}) together imply that there exists 
$$a_2\in \supp(r_{a_1} w_0\eta) =\supp(w_0\eta)$$ 
with $(r_{a_1} w_0x_i, a_2)>0$. Then 
$$\dep(r_{a_2} r_{a_1} w_0 x_i)=\dep(r_{a_1}w_0 x_i)-1 =\dep(w_0x_i)-2.$$
Repeat this process, and after a finite number of iterations, we shall see that 
$$r_{a_{k-1}} \cdots r_{a_1} w_0 x_i =a_k\in\supp(w_0\eta) \subset \Pi,$$
where $a_1, \ldots, a_{k-1}\in \supp(w_0 \eta)$. This together with (\ref{conn}) enables us to conclude that in either case of $\pos(w_0\eta)=\emptyset$ or $\pos(w_0\eta)\neq\emptyset$, by setting $N=\max\{N_1, N_2\}$,
$$\supp(w_0 x_i)\subseteq \supp(w_0\eta), \quad \forall i\geq N,$$
and therefore for all $i\geq N$
\begin{align*}
\supp(x_i) &\subseteq \supp(w_0 x_i)\cup I\\
           &\subseteq\supp(w_0 \eta)\cup \supp(\eta)\\
           &= \supp( \eta),
\end{align*}
whence $\supp(x_i) =\supp(\eta)$ for all $i\geq N$ for some $N\in \N$ in the case when $\pos(\eta)\neq \emptyset$ as well.
\end{proof}

\begin{remark}
\label{rmk}
Suppose that $(x_i)_{i\in \N}\subset \Phi^+$ is an infinite injective sequence of positive roots totally ordered by dominance, that is, 
$$x_{i+1}\dom x_{i}, \text{ and } x_{i+1}\neq x_i,\quad \forall i\in \N.$$
Suppose further that $\lim_{i\to \infty}\widehat{x_i}=\eta$. Then for each $j\in \N$, the continuity of the bilinear form $(\,,\,)$ ensures that
$$(x_j, \eta)=\lim_{i\to \infty}(x_j, \widehat{x_i})\geq 0.$$
Now Proposition~\ref{domU} yields that $x_{i+1}-x_i \in U^*$ for all $i\in \N$, which in turn implies that $\supp(x_i)\subseteq \supp(x_{i+1})$ for all $i\in \N$. This nested property of the successive supports along a dominance chain and position~\ref{suppcon} together establish the following: 
\end{remark}

\begin{theorem}
  \label{support}
  Suppose that $(x_i)_{i\in \N}\subset \Phi^+$ is an infinite injective sequence of positive roots totally ordered by dominance, that is, 
$$x_{i+1}\dom x_{i}, \text{ and } x_{i+1}\neq x_i,\quad \forall i\in \N.$$
Suppose further that $\eta$ is an accumulation point for the normalised sequence $(\widehat{x_i})_{i\in \N}$. Then there exists some $N\in \N$ with 
$$\supp(x_i)=\supp(\eta), \quad \forall i\geq N.$$
In particular, all possible accumulation points for the normalised sequence $(\widehat{x_i})_{i\in \N}$ have the same connected support.\qed 
\end{theorem}

\begin{lemma}
\label{domineq}
Suppose that $x\dom y\dom z$. Then
$$(x, z)\geq (x, y)\text{ and } (x, z)\geq (y,z).$$
\end{lemma}
\begin{proof}
Note that the claimed result certainly holds if either $x=y$ or $y=z$. Hence we may assume that $x\neq y\neq z$.
Since $x\dom y$, it follows from \cite[Proposition 4.11]{FU2} that there exists some $w\in W$ with
$wx\in \Phi^+$ and $wy\in \Phi^-$ and moreover
$$(w(x-y), a)\leq 0,\,\,\forall a\in \Pi.$$ 
Now since $y\dom z$, it follows from $wy\in \Phi^-$ that $w z\in \Phi^-$ too, and consequently, 
$$(x-y, z)=(w(x-y), wz)\geq 0.$$
Whence $(x, z)\geq (y, z)$.

In an entirely similar manner, one can prove that $(x, z)\geq (x, y)$.
 \end{proof}

\begin{proposition}
\label{eq1}
Suppose that $x\dom y\dom z$ with $x\neq y\neq z$. 
\begin{itemize}
\item[(i)]If $(x, z)=(x, y)$ then 
$(x, z)=(x, y)=(y, z)=1$,
and the reflection subgroup $W':=\langle r_x, r_y, r_z\rangle$ is affine dihedral. In particular, $x\in \R y\oplus \R z$. 
\item[(ii)]If $(x, z)=(y, z)$ then 
$(x, z)=(x, y)=(y, z)=1$,
and the reflection subgroup $W':=\langle r_x, r_y, r_z\rangle$ is affine dihedral. In particular, $x\in \R y\oplus \R z$.
\end{itemize}
\end{proposition}
\begin{proof}
We only need to prove (i), for an entirely similar proof also applies to (ii).

For (i), first, note that $W'$ is irreducible as a Coxeter group. Indeed, the nonzero values of 
$(x, y)$, $(x, z)$ and $(y, z)$, as ensured by the dominance between these three generating roots, imply that these three roots belong to the same irreducible component of the root subsystem of $W'$.
Next, since $x\dom y\dom z$, it follows from  \cite[Proposition 4.10]{FU2} that $x-y\in  U^*$ and $y-z\in U^*$.
Then Proposition~\ref{leq0} yields that $(x-y, y-z)\leq 0$. That is,
$$0\geq (x-y, y-z)=\underbrace{(x, y-z)}_{=(x, y)-(x,z)=0}-\overbrace{(y, y-z)}^{=1-(y, z)\leq 0},$$ 
and consequently, $(y, y-z)=0$, which in turn gives $(y, z)=1$. Then 
$$(x, y-z)=0, \quad (y, y-z)=0 \quad\text{and}\quad(z, y-z)=0.$$
That is $y-z$ is in the radical of the bilinear form $(\,,\,)$ restricted to the subspace spanned by $x$, $y$ and $z$.
Furthermore, in the irreducible reflection subgroup $W'$
$$w\cdot (\widehat{y-z})=\widehat{y-z}, \quad \forall w\in W',$$
which in turn establishes that $E(W')=\{\widehat{y-z}\}$ is a singleton (following \cite[Theorem 3.1]{DHR13}).
Now suppose for a contradiction that $W'$ is non-affine. Then it follows from \cite[Proposition 4.12]{FRX} that there exists a non-affine dihedral reflection subgroup $W''\subset W'$ with $|E(W'')|=2$, contradicting $E(W'')\subseteq E(W')$. Consequently, $W'$ is affine, whence 
$$(x, y)=(y, z)=(x, z)=1.$$
Then $(x-y, x)=(x-y, y)=(x-y, z)=0$, and $x-y$ is also in the radical of the bilinear form $(\,,\,)$ restricted to the subspace spanned by $x$, $y$ and $z$. Now if $W'$ is of rank $3$, then the set $\{x, y, z\}$ is linearly independent, and the radical of the bilinear form $(\,,\,)$ restricted to the subspace spanned by $x$, $y$ and $z$ is of dimension $2$ spanned by $x-y$ and $y-z$. On the other hand, it is well known (for instance, see \cite[Section 2.6]{HM}) that the radical of the restriction of the bilinear form $(\,,\,)$ on the subspace spanned by $x, y$ and $z$ is $1$-dimensional if $W'$ is irreducible affine, and we have a contradiction, whence $W'$ is dihedral.  
\end{proof}

In fact, Lemma~\ref{domineq} can be strengthened to the following:
\begin{proposition}
\label{geq}
Suppose that $x\dom y\dom z$. Then 
$$(x, z)\geq (x, y)(y, z).$$
\end{proposition}
\begin{proof}
Note that the claimed result trivially holds if either $x=y$ or $y=z$. Thus we may assume that $x\neq y\neq z$.
We have the following two possibilities:
\begin{itemize}
  \item[(1)] $(x, -r_y z)\geq 1$, or else
  \item[(2)] $(x, -r_{y}z)<1$.
\end{itemize}
If (1) is the case, then either 
\begin{itemize}
\item[(1.1)] $x\dom -r_y z$, or else
\item[(1.2)] $-r_y z\dom x$.
\end{itemize}
Note that $(y, -r_y z)=(y, z)\geq 1$ and hence either $y\dom -r_y z$ or $-r_y z\dom y$. If $y\dom -r_y z$ then $-y=r_y y\dom r_y (-r_y z)=-z$ which implies that $z\dom y$, a contradiction. Therefore, $-r_y z\dom y$.
Now, if (1.1) is the case, then $x\dom -r_y z \dom y$, and hence it follows from Lemma~\ref{domineq} that 
\begin{align*}
(x, y)&\geq (x, -r_y z)=(x, -z+2(y, z)y)\\
      &=-(x, z)+2(y, z)(x,y),
\end{align*}
and hence $(x, z)\geq 2(x, y)(y, z)-(x, y)\geq (x, y) (y, z)$, as required.

On the other hand, if (1.2) is the case, then 
$$-r_y z\dom x\dom y,$$
and hence it follows from Lemma~\ref{domineq} that  $(-r_y z, y)\geq (-r_y z, x)$. That is, 
\begin{align*}
(y, z)&\geq (-z+2(y, z)y, x)=-(x, z)+2(y, z)(x, y),\\
\noalign{\hbox{and hence}}
(x, z)&\geq 2(x, y)(y, z)-(y, z)\geq (x, y)(y, z),
\end{align*}
as required.

Finally, if (2) is the case, then
$$(x, -r_y z)=(x, -z)+2(y, z)(x, y)<1,$$
that is
$$-(x, z)+2(y, z)(x, y)<1,$$
and thus
$$2(x, y)(y, z)-1 <(x, z),$$
and consequently
$$(x, y)(y, z) \leq (x, y)(y, z)+\underbrace{((x, y)(y, z)-1)}_\text{$\geq 0$}=2(x, y)(y, z)-1<(x, z),$$
as required.
\end{proof}

The next result is an easy observation which may help to simplify arguments later on:

\begin{lemma}
\label{orth}
Suppose that $\eta_1, \eta_2\in U^*$ with $(\eta_1, \eta_2)=0$. Then for each $x\in \Phi$ we have
$$(\eta_1, x)(\eta_2, x)\geq 0.$$
\end{lemma}
\begin{proof}
Since $U^*$ is $W$-invariant, it follows that $w \eta_1\in U^*$ for all $w\in W$, and in particular, for each $x\in \Phi$ we have
$r_x \eta_1\in U^*$.
Then 
\begin{align*}
0\geq (r_x \eta_1, \eta_2)&=(\eta_1-2(\eta_1, x)x, \eta_2)\\
                          &=\underbrace{(\eta_1, \eta_2)}_\text{$=0$}-2(\eta_1, x)(\eta_2, x).
\end{align*}
Consequently, $(\eta_1, x)(\eta_2, x)\geq 0$.
\end{proof}

\begin{lemma}
\label{uineq}
Suppose that $x$ and $y$ are roots with $x\dom y$. Then for each $\eta\in U^*$ we have
$$(\eta, y)\geq (\eta, x).$$
\end{lemma}
\begin{proof}
Recall that $x\dom y$ if and only if $x-y\in U^*$ (by \cite[Proposition 4.10]{FU2}). Hence for all $\eta\in U^*$
$$(\eta, x-y)\leq 0.$$
Whence $(\eta, y )\geq (\eta, x)$.
\end{proof}

%

\begin{proposition}
\label{aff1}
Let $(x_i)_{i\in \N}\subset \Phi^+$ be an infinite injective sequence of positive roots totally ordered by dominance. Suppose that $\eta_1$ and $\eta_2$ are accumulation points of the normalised sequence $(\widehat{x_i})_{i\in \N}$ where 
$|\pos(\eta_1)|<\infty$. Then $\eta_1=\eta_2$.
\end{proposition}
\begin{proof}
Choose two subsequences $(y_i)_{i\in \N}$ and $(z_i)_{i\in \N}$ of $(x_i)_{i\in \N}$ such that $\lim_{i\to\infty}\widehat{y_i}=\eta_1$ and 
$\lim_{i\to\infty}\widehat{z_i}=\eta_2$, and for all $i\in \N$ 
$$z_{i+1}\dom y_{i+1}\dom z_i \quad\text{($z_{i+1}\neq y_{i+1}\neq z_i$)}.$$
Then Theorem~\ref{support} yields that there exists some $N\in \N$ such that
\begin{equation}
\label{equalsupp}
\supp(y_i) =\supp(z_i) =\supp(\eta_1), \quad \forall i\geq N, 
\end{equation}
and, in particular, $\supp(\eta_1)$ is connected.

Next, note that for each $w\in W$, there exists some $N_w\in \N$ such that $(wx_i)_{i\geq N_w}$ is an infinite injective sequence of positive roots  totally ordered by dominance, and $w\cdot \eta_1$ and $w\cdot \eta_2$ are accumulation points of $(\widehat{w  x_i})_{i\geq N_w}$. Further
$(w y_i)_{i\in \N}$ and $(w z_i)_{i\in \N}$ are subsequences of $(w x_i)_{i\in \N}$ where 
$$\lim_{i\to \infty}\widehat{w y_i}= w\cdot\eta_1 \quad \text{and}\quad \lim_{i\to \infty}\widehat{w z_i}=w\cdot\eta_2, $$
with $w z_{i+1}\dom w y_{i+1}\dom w z_i$ for all $i\in \N$. Then Theorem~\ref{support} again yields that $\supp(w\cdot\eta_1)$ is connected for each $w\in W$. 
Now the connectedness of $\supp(w\cdot\eta_1)$ (for each $w\in W$) and the fact $|\pos(\eta_1)|<\infty$ together imply that $\{\eta_1\}=E(W')$ for some irreducible affine reflection subgroup $W'$ of $W$ (that is, $\eta_1$ is the sole limit root in the normalised root subsystem for the reflection subgroup $W'$) by \cite[Theorem 5.23]{FRX}. Then \cite[Theorem 5.25]{FRX} yields that there exists $w\in W$ such that 
$$\{w\cdot\eta_1\}=E(W_I),$$
for some irreducible affine standard parabolic subgroup $W_I$ of $W$, where $I=\supp(w\cdot \eta_1)$.
On the other hand, it follows from Theorem~\ref{support} that there exists some $N'\in \N$ such that
$$\supp(w y_i)= \supp(w z_i)=\supp(w\cdot\eta_1), \quad\forall i\geq N',$$
and thus $I=\supp(w\cdot \eta_2) =\supp(\lim_{i\to\infty} w\cdot z_i)$. Now since $W_I$ is irreducible affine, it follows that $\widehat{E\cap \R I} =\{ w\cdot \eta_1\}$, and therefore $w\cdot \eta_2=w\cdot \eta_1$,  whence $\eta_1=\eta_2$.
\end{proof}

\begin{proposition}
\label{forminf}
Suppose that $(x_i)_{i\in \N}\subset \Phi^+$ is an infinite injective sequence of positive roots totally ordered by dominance, that is, 
$$x_{i+1}\dom x_i,\text{ and } x_{i+1}\neq x_i,\quad \forall i\in \N.$$
Furthermore, suppose that for each accumulation point $\eta$ of the normalised sequence $(\widehat{x_i})_{i\in \N}$ we have
$$|\pos(\eta)|=\infty.$$ 
Then 
$$\lim_{i\to \infty} (x_i, x_1)=\infty.$$
\end{proposition}
\begin{proof}
Suppose, to the contrary, that the set $\big\{\,(x_i, x_1)\mid i\in \N\,\big\}$ is bounded above.  Note that Proposition~\ref{geq} ensures that 
$$(x_{i+1}, x_1)\geq (x_i, x_1)\geq 1, \quad \forall i\in \N.$$
Thus the sequence of bilinear form values $\big((x_i, x_1)\big)_{i\in\N}$ is a weakly increasing sequence of real numbers, and is bounded above by our assumption. Therefore
$$\lim_{i\to \infty} (x_i, x_1)=M, \quad \exists M\geq 1.$$
Now it follows from \cite[Lemma 4.3]{Dyer12} that the following set of bilinear form values
$$\big\{\,(x, y)\mid \text{ $x,y\in \Phi$ and } 1\leq (x, y)\leq M\,\big\}$$
is finite. Hence passing to a subsequence $(x_{i_k})$ if necessary, we may assume that
\begin{equation*}\label{M}
(x_{i_k}, x_1)=M,\quad \forall i_k\geq 2,  
\end{equation*}
and, in particular, Proposition~\ref{eq1} implies that there exists some $j\geq 2$ with
$$(x_{i_k}, x_1)=(x_j, x_1)=1, \quad \forall i_k\geq j.$$
But then Part (ii) of Proposition~\ref{eq1} yields that
$$\langle r_{x_{i_k}}, r_{x_j}, r_{x_1}\rangle, \quad \forall i_k\geq j$$ 
is affine dihedral and each $x_{i_k}\in V':= \R x_{j}\oplus \R x_1$, for all $i_k\geq j$. Note that in this $2$-dimensional subspace $V'=\R x_{j}\oplus \R x_1$, we have
$$Q\cap V' =\rad(V'),$$
and it is a $1$-dimensional subspace by \cite[Proposition 2.6]{HM}.
Now let $\{\eta\}:= \widehat{Q\cap V'}=Q\cap V'\cap V_1$. Since an accumulation point must be isotropic and is contained in $V_1$ (by \cite[Theorem 2.7]{HLR11}), it follows that the only possible accumulation point in the subspace $V'$ is contained in 
$\widehat{Q\cap V'}=Q\cap V'\cap V_1=\{\eta\}$,  and hence the normalised subsequence 
$(\widehat{x_{i_k}})$ is convergent with  $\eta$ being the limit. 
Since $\eta$ is also the limit root arising from the affine dihedral reflection subgroup $\langle r_{x_1}, r_{x_j}\rangle$, it follows from \cite[Theorem 5.5]{FRX}
that 
$$|\pos(\eta)|<\infty,$$
a contradiction. 
\end{proof}

\begin{lemma}
  \label{tech1}
  Let $\eta\in U^*$ and let $x\in \Phi^+$ with $(\eta, x)\geq 0$. Now let $y\in \pos(\eta)$. Then 
  $(x, y)>-1$. Moreover, if $(\eta, x)=0$, then either $|(x, y)|<1$ or $x\dom y$.
\end{lemma}
\begin{proof}
  Let $W':=\langle r_x, r_y\rangle$. Set $\Phi(W')$ to be the \emph{root subsystem} for $W'$, 
  that is, 
  $$\Phi(W'):= \{x\in \Phi\mid r_x\in W'\}.$$
  Direct dihedral calculations show that if $(x, y)<-1$ then  
  $$\Phi(W')=\big\{ c_{i\pm 1} x+c_{i} y\mid  c_i:=\frac{\sinh i\theta}{\sinh \theta}, i\in \Z,\, \, \theta:=\cosh^{-1}(-(x, y))\big\};$$
  while if $(x, y)=-1$ then
  $$\Phi(W')=\big\{ c_{i\pm 1} x+c_{i} y\mid i\in \Z,\, c_i:=i \big\},$$
  and, in particular, $\Phi(W')$ is an infinite root subsystem.
  Then note that $$0<\epsilon:=(\eta, y)\leq (\eta, z),\quad \forall z\in \Phi^+(W')\setminus\{x\}, $$
  and consequently, the following set is infinite
  $$\{\,z\in \Phi^+(W')\mid (\eta, z)\geq \epsilon \,\},$$
  a contradiction to Proposition~\ref{epsilon}. Whence $(x, y)>-1$.
  
  Next, suppose that $(\eta, x)=0$.  If $(x, y)\geq 1$, then there is dominance between $x$ and $y$. If $y\dom x$ 
  then $y-x\in U^*$, and hence 
  $$0\geq (\eta, y-x)=(\eta, y)-(\eta, x)=(\eta, y)>0,$$
  a contradiction, whence $x\dom y$.
\end{proof}

\begin{proposition}
\label{posinf}
Suppose that $(x_i)_{i\in \N}\subset \Phi^+$ is an infinite injective sequence  with 
$$x_{i+1}\dom x_i,\,\,\forall i\in \N,\text{ and } \lim_{i\to\infty}\widehat{x_i}=\eta.$$
If $|\pos(\eta)|=\infty$, then 
$$(\eta, x_i)>0, \quad \forall i\in \N.$$
\end{proposition}
\begin{proof}
Note that we only need to prove that 
\begin{equation}\label{P1}
  (\eta, x_1)>0.
\end{equation}
Indeed, the subsequence $(x_i)_{i\in \N_{\geq 2}}$ clearly satisfies all the requirements of $(x_i)_{i\in \N}$, and if (\ref{P1}) holds then by replacing $(x_i)_{i\in \N}$ by $(x_i)_{i\in \N_{\geq 2}}$ we shall get 
$$(\eta, x_2)>0,$$
and so on. Thus if (\ref{P1}) holds then 
$$(\eta, x_i)>0\quad \forall i\in \N.$$
Therefore, it remains to prove (\ref{P1}).
Note also that 
$$|\pos(w\eta)|=\infty, \quad \forall w\in W,$$
and hence we may assume that $x_1\in \Pi$.
Now suppose for a contradiction that $(\eta, x_1)=0$ (recall that $(x_1, \eta)\geq 0$ by Remark~\ref{rmk}). Since $x_i\dom x_1$ for all $i\in \N$, it follows from Lemma~\ref{uineq} that
$$(\eta, x_1)\geq (\eta, x_i), \quad \forall i\in \N.$$
Consequently, the assumption of $(\eta, x_1)=0$ forces  
\begin{equation}\label{always}
  (\eta, x_i)=0,\quad \forall i\in \N.
\end{equation}
On the other hand, Proposition~\ref{forminf} yields that
$$\lim_{i\to\infty}(x_i, x_1)=\infty.$$
Thus, there exists some $N\in \N$ with 
\begin{equation}\label{L}
  (x_N, x_1)>\frac{1}{L},
\end{equation}
where $L:=\min\{\,|(x, y)|\colon \text{ $x, y\in\Phi$ and $(x, y)\neq 0$ } \,\}$.
Recall that $L>0$ by Proposition~\ref{lowbd}. 

Next let $\mathbf{a}=(a_i)_{i\in \N}$ be a transitional sequence for $\eta$ with corresponding limit $\eta_{\mathbf{a}}$. For each $i\in \N$ we set
$$a_{i,1}:=r_{a_1}\cdots r_{a_{i-1}}a_i.$$
Recall that Lemma~\ref{pos} yields that $a_{i, 1}\in \Phi^+$. Next, for each $i\in \N$ we set
$$\beta_i:=r_{a_{i-1}}\cdots r_{a_1} x_N.$$
Recall that we have established in Lemma~\ref{pos} that for all $i\in \N$
$$N(r_{a_{i-1}}\cdots r_{a_1})\subset N(r_{a_i}\cdots r_{a_1}) =\{\,a_{j,1}\mid j\leq i \,\}\subset \pos(\eta).$$
Thus  
\begin{equation}\label{beta}
  \beta_i\in \Phi^+,\quad \forall i\in \N,
\end{equation}
for otherwise 
$$x_N\in N(r_{a_{i-1}}\cdots r_{a_1})\subset \pos(\eta),$$
contradicting $(\eta, x_N)=0$.

Now suppose for a contradiction that there exists some $i\geq 1$ with 
$$(\beta_i, a_i)<0.$$
Then
$$(x_N, a_{i,1})=(x_N, r_{a_1}\cdots r_{a_{i-1}}a_i)=(r_{a_{i-1}}\cdots r_{a_1}x_N, a_i)=(\beta_i, a_i)<0.$$
and in particular, 
$$r_{x_N}a_{i, 1}\in \Phi^+,$$
and 
\begin{equation}\label{-L}
  (x_N, a_{i,1})\leq -L,
\end{equation}
by the definition of $L$.
Combine the fact $a_{i, 1}\in \pos(\eta)$, for all $i\in \N$, together with our assumption of $(x_1, \eta)=0$, we deduce from Lemma~\ref{tech1} that
\begin{equation}\label{>-1}
(x_1, a_{i, 1})>-1,  
\end{equation}
and 
$$|(x_1, a_{i,1})|<1,$$
for otherwise $x_1\dom a_{i,1}$, contradicting that $x_1\in \Pi\subset D_0$.
Moreover, observe that in addition to $r_{x_N}a_{i,1}\in \Phi^+ $ we have $r_{x_N}a_{i,1}\in \pos(\eta)$. Indeed,
\begin{align*}
(\eta, r_{x_N} a_{i, 1}) &=(\eta, a_{i,1}-2(a_{i,1}, x_N) x_N)\\
                         &=(\eta, a_{i, 1}) -2(a_{i,1}, x_N) \underbrace{(\eta, x_N)}_\text{$=0$ by (\ref{always})}\\
                         &=(\eta, a_{i, 1})>0.
\end{align*}
Now 
\begin{align*}
(r_{x_N} a_{i,1}, x_1)&=(a_{i,1}-2(x_N, a_{i,1})x_N, x_1)\\
                      &=\underbrace{(a_{i, 1}, x_1)}_\text{$>-1$ by(\ref{>-1})}-2\underbrace{(x_N, a_{i,1})}_\text{$\leq -L$ by (\ref{-L})}\overbrace{(x_N, x_1)}^\text{$\geq \frac{1}{L}$ by (\ref{L})}\\
                      &>1.
\end{align*}
But now we have $r_{x_N}a_{i,1}\in \pos(\eta)$ and $(x_1,\eta)=0$ and $(r_{x_N} a_{i,1}, x_1)>1$, and therefore Lemma~\ref{tech1} forces
$$x_1\dom r_{x_N} a_{i,1},$$
contradicting $x_1\in \Pi\subset D_0$. Consequently, 
\begin{equation}\label{mono}
(\beta_i, a_i)\geq 0,\quad \forall i\in \N.
\end{equation}
Thus, the sequence of depths  $(\dep(\beta_i))_{i\in \N}$ is weakly decreasing. 
Let $m_{\mathbf{a}}$ be a separating index for $\mathbf{a}=(a_i)_{i\in \N}$. Since there are only finitely many positive roots of depth less than or equal to $\dep(\beta_{m_{\mathbf{a}}})$, and we have established in (\ref{beta}) that $\beta_i\in \Phi^+$ for all $i\in \N$, in view of (\ref{mono}) it follows that there exists some $m\geq m_{\mathbf{a}}$ with 
$$(\beta_m, a_i) =0,\quad \forall i\geq m_{\mathbf{a}}.$$
But on the other hand, note that the sequence $\big((r_{a_{m-1}}\cdots r_{a_1} x_i)\big)_{i\in \N}$ is an infinite injective sequence totally ordered by dominance whose corresponding normalised sequence $\big(\widehat{(r_{a_{m-1}}\cdots r_{a_1} x_i)}\big)_{i\in \N}$ converges to $\widehat{r_{a_{m-1}}\cdots r_{a_1} \eta}$. Now $r_{a_{m-1}}\cdots r_{a_1}$ can only make finitely many positive roots negative, hence starting the sequence $\big((r_{a_{m-1}}\cdots r_{a_1} x_i)\big)_{i\in \N}$ at some $i_0>1$ if needed, we may assume that the following sequence 
$$\big((r_{a_{m-1}}\cdots r_{a_1} x_i)\big)_{i>i_0}$$ 
consists of all positive roots. Since in any dominance chain, the sequence of supports is weakly increasing and eventually stabilises  to the support of any accumulation point, it follows that  
\begin{align*}
\supp(\beta_m) &=\supp(r_{a_{m-1}}\cdots r_{a_1} x_N)
\subseteq \supp(r_{a_{m-1}}\cdots r_{a_1}\eta)\\
&=\underbrace{\supp(\eta_m)=\supp(\eta_{m_{\mathbf{a}}})}_\text{by Lemma~\ref{stn}}\\
&=\mathbf{a}_{\infty}, 
\end{align*}
by Theorem~\ref{zero2}, since $\supp(\eta_{m_{\mathbf{a}}})$ is connected (indeed, note that each $\widehat{\eta_j}=\widehat{r_{a_{j-1}}\cdots r_{a_1}\eta}$ is the limit for the sequence of $\big(\widehat{(r_{a_{j-1}}\cdots r_{a_1} x_i)} \big )_{i\in \N}$, and therefore has connected support following Theorem~\ref{support}, and, in particular, this works when $j=m_{\mathbf{a}}$). 
But then 
$$(\beta_m, \beta_m)=0,$$ contradicting $\beta_m\in \Phi$, completing the proof.
\end{proof}

\begin{theorem}
\label{relation}
Suppose that $(x_i)_{i\in \N}\subset \Phi^+$ and $(y_i)_{i\in \N}\subseteq \Phi^+$ are two infinite injective sequences of positive roots totally ordered by dominance. Moreover, suppose that 
$$x_{i+1}\dom y_{i+1}\dom x_i\dom y_i, \quad \forall i\in \N,$$
and
$$\lim_{i\to\infty} \widehat{x_i}=\eta_1 \text{ and }\lim_{i\to\infty}\widehat{y_i}=\eta_2.$$ 
Then 
\begin{itemize}
\item[(i)] Let $x\in \Phi$. Then $(\eta_1, x)>0$ (respectively, $<0$ and $=0$) if and only if 
$(\eta_2, x)>0$ (respectively, $<0$ and $=0$).
\item[(ii)] Let $\zeta\in U^*$. Then $(\zeta, \eta_1)=0$ (respectively, $<0$) if and only if $(\zeta, \eta_2)=0$ (respectively, $<0$).
\end{itemize}
\end{theorem}
\begin{proof}
Note that in view of Proposition~\ref{aff1} we may safely assume that $|\pos(\eta_1)|=\infty=|\pos(\eta_2)|$.

(i): If $x\in \pos(\eta_1)$, then Lemma~\ref{normbd} implies that
$$(x, x_i)=(x, \widehat{x_i})|x_i|\to \infty, \text{ as $i\to\infty$}.$$ Thus  
$(x, x_i)\geq 1$ for infinitely many $i$'s, that is, there is dominance between $x$ and infinitely many $x_i$'s. Then Lemma~\ref{dfin} yields that
$x$ can only dominate finitely many $x_i$'s. Consequently, $x$ has to be dominated by some $x_i$ once $i$ is sufficiently large, that is, $x\in D(x_i)$ for some $i\in \N$. Now because $y_{i+1}\dom x_i$, we also have $x\in D(y_{i+1})$. Whence Lemma~\ref{uineq} yields that 
$$(\eta_2, x)\geq \underbrace{(\eta_2, y_{i+1})>0}_\text{by Proposition \ref{posinf}},$$
and hence $\pos(\eta_1)\subseteq \pos(\eta_2)$.

By symmetry, we also have $\pos(\eta_2)\subseteq \pos(\eta_1)$, whence 
$$\pos(\eta_1)=\pos(\eta_2).$$

Next, suppose that $x\in \Phi^+$ and $(x, \eta_1)<0$. Then $(r_x\cdot \eta_1, x)>0$, and the preceding proof establishes that 
$(r_x\cdot \eta_2, x)>0$, whence $(x, \eta_2)<0$. By symmetry, we have for all $x\in \Phi^+$
$$(\eta_1, x)<0 \text{ if and only if } (\eta_2, x)<0.$$

Upon combining the above we have that for all $x\in \Phi^+$
$$(\eta_1, x)=0 \text{ if and only if }(\eta_2, x)=0.$$

Combining all of the above,  we  have 
$(\eta_1, x)>0$ (respectively, $<0$ and $=0$) if and only if 
$(\eta_2, x)>0$ (respectively, $<0$ and $=0$), for each $x\in \Phi$.

(ii): Suppose that $\zeta\in U^*$ and $(\zeta, \eta_1)=0$. Then Lemma~\ref{orth} yields that for all $i\in\N$
$$(\zeta, x_i)(\eta_1, x_i)\geq 0.$$
But we know from Proposition~\ref{posinf} that $(\eta_1, x_i)>0$ for all $i\in \N$, and therefore
$$(\zeta, x_i)\geq 0, \quad \forall i\in \N.$$
Now given the fact that $x_i\dom y_i$ for each $i\in \N$, Lemma~\ref{uineq} yields that
$$(\zeta, y_i)\geq (\zeta, x_i)\geq 0, \quad \forall i\in \N.$$
Note that this means $(\zeta, \eta_2)\geq 0$, but given that $\zeta$ and $\eta_2$ are both in $U^*$, we deduce that
$$(\zeta, \eta_2)=0.$$

From the above we also conclude that for $\zeta\in U^*$ 
$$(\zeta, \eta_1)<0 \text{ if and only if } (\zeta, \eta_2)<0.$$

\end{proof}




\begin{definition}
\label{extreme}
Let $\Ext$ denote the set of extreme points of $\overline Z$. Thus
$\eta\in\Ext$ if, whenever
\[
\eta=t\eta_1+(1-t)\eta_2
\]
with $\eta_1,\eta_2\in\overline Z$ and $0<t<1$, one has
$\eta_1=\eta_2=\eta$.
\end{definition}

\begin{lemma}
\label{norm}
Let $\eta\in\widehat{U^*}=U^*\cap V_1$ and $x\in \Phi$. Then
$$(x, \eta)|x|_1 <\frac{1}{2}.$$ 
\end{lemma}

\begin{proof}
Since $U^*$ is $W$-invariant,
\[
r_x\eta\in U^*\setminus\{0\}\subseteq\PLC(\Pi).
\]
Therefore
\[
0<|r_x\eta|_1
  =|\eta|_1-2(x,\eta)|x|_1
  =1-2(x,\eta)|x|_1,
\]
and hence
\[
(x,\eta)|x|_1<\frac12.
\]
\end{proof}


Note that Lemma~\ref{norm} implies that for $\eta\in \overline{Z}$, and any sequence $(x_i)_{i\in \N}\subset\Phi^+$, the sequence $(\eta, x_i)|x_i|_1$ is bounded above and its limsup is at most 1/2.

\begin{proposition}
\label{limsup}
Suppose that $(x_i)_{i\in \N}\subset\Phi^+$ is an infinite injective sequence of positive roots totally ordered by dominance 
with $$\lim_{i\to\infty}\widehat{x_i}=\eta.$$ Moreover suppose that $\eta'\in \Ext$ with 
$$\limsup_{i\to\infty} (\eta', x_i)|x_i|_1>0.$$
Then $\eta'=\eta$.
\end{proposition}
\begin{proof}
Suppose, for a contradiction, that $\eta\neq \eta'$. 
Note that by passing to a subsequence of $(x_i)_{i\in \N}$ we may assume that 
$$\lim_{i\to \infty} (\eta', x_i)|x_i|_1 =M >0.$$
From Lemma~\ref{norm} we may deduce that $M\leq \frac{1}{2}$.

If $M=\frac{1}{2}$, then passing to a further subsequence of $(x_i)$ if necessary, we may assume that 
there exists some $N\in \N_{>1}$  such that 
$$(\eta', x_i)|x_i|_1 \geq \frac{1}{2}-\frac{1}{2^i}, \quad \forall i\geq N.$$
Then for all $i\geq N$
\begin{equation}\label{convPI}
r_{x_i}\cdot \eta' =\lambda_i\eta'+ \mu_i\widehat{x_i}\in \widehat{U^*}\subset \conv(\Pi),  
\end{equation}
where $\lambda_i =\frac{1}{1-2(\eta', x_i)|x_i|_1} \geq 2^{i-1}$ and $\mu_i=\frac{-2(\eta', x_i)|x_i|_1}{1-2(\eta', x_i)|x_i|_1}\leq 1-2^{i-1}$ with $\lambda_i+\mu_i=1$.
Recall that $r_{x_i}\cdot \eta'$ lies on the line passing through $\eta'$ and $\widehat{x_i}$, and since $\mu_i$ is negative, it follows that $r_{x_i}\cdot\eta'$ lies on that portion of the line where $\eta'$ is strictly in between $r_{x_i}\cdot\eta'$ and $\widehat{x_i}$. Observe that as  the value of $\lambda_i$ increases, the distance between $r_{x_i}\cdot\eta'$ and $\eta'$ increases (unless $\eta=\eta'$).  Now given that  $\lambda_i$ is unbounded when $i$ is unbounded, it is clear that the sequence $(r_{x_i}\cdot\eta')_{i\in \N}$ is not contained in a bounded region of $\conv(\Pi)$,   
contradicting $\conv(\Pi)$ being bounded. 

Consequently, $M<\frac{1}{2}$. Then 
$$\eta'':=\lim_{i\to\infty} r_{x_i}\cdot \eta'=\underbrace{\frac{1}{1-2M}}_\text{$>0$}\eta'+\underbrace{\frac{-2M}{1-2M}}_\text{$<0$}\eta.$$
In particular, $\eta''\in\overline{Z}=\conv(E)$, since $\overline{Z}$ is compact. Note that $\eta''$ lies on the line passing through $\eta$ and $\eta'$. Now since the coefficient of $\eta$ in $\eta''$ is strictly negative while the coefficient of $\eta'$ in $\eta''$ is strictly positive, it follows that $\eta'$ lies strictly between $\eta\in E\subset \overline{Z}$ and $\eta''\in \overline{Z}$, contradicting $\eta'\in \Ext$, unless $\eta=\eta'$.  
\end{proof}

\begin{remark}
It can be verified with direct rank-$2$ calculations that if $W$ is an infinite non-affine dihedral group, and $(x_i)_{i\in \N}$  is an infinite injective sequence of positive roots of $W$ totally ordered by dominance, then $(\widehat{x_i})_{i\in \N}$ converges. 
In fact, if $\eta =\lim_{i\to \infty} \widehat{x_i}$, then 
$$\lim_{i\to \infty}(\eta, x_i)|x_i|_1 =\frac{1}{2}.$$ 
\end{remark}

\begin{corollary}
\label{convexCom}
Suppose that $(x_i)_{i\in \N}\subset\Phi^+$ is an infinite injective sequence of positive roots totally ordered by dominance 
with $$\lim_{i\to\infty}\widehat{x_i}=\eta.$$ Moreover suppose that there exists some $\eta'\in \overline{Z}$ with 
$$\limsup_{i\to\infty} (\eta', x_i)|x_i|_1>0.$$ 
Then $\eta\in \Ext$ and 
$$\limsup_{i\to\infty} (\eta, x_i)|x_i|_1>0.$$
\end{corollary}
\begin{proof}
Note that if $\eta'\in \Ext$ then the desired result follows from Proposition~\ref{limsup}. Hence we may assume that $\eta'\notin\Ext$, and 
$$\eta'=\sum_{j=1}^m k_j \eta_j,$$
where $\eta_1, \ldots, \eta_m\in \Ext$ and $k_1, \ldots, k_m>0$ with $k_1+\cdots k_m=1$. Then the assumption implies that there exists some $j\in \{1, \ldots, m\}$ with
$$\limsup_{i\to\infty} (\eta_j, x_i)|x_i|_1>0.$$
Then Proposition~\ref{limsup} ensures that $\eta=\eta_j$ and we are done.
\end{proof}

\begin{corollary}
\label{limFormNorm}
Suppose that $(x_i)_{i\in \N}\subset\Phi^+$ is an infinite injective sequence of positive roots totally ordered by dominance 
with $$\lim_{i\to\infty}\widehat{x_i}=\eta \text{ and } |\pos(\eta)|=\infty.$$
Moreover suppose that 
$$\lim_{i\to\infty} (\eta, x_i)|x_i|_1 =0.$$
Then 
$$\lim_{i\to\infty} (\eta', x_i)|x_i|_1 =0,$$
for all $\eta'\in \overline{Z}$ with $(\eta, \eta')=0$.
\end{corollary}
\begin{proof}
Since $|\pos(\eta)|=\infty$, it follows from Proposition~\ref{posinf} that 
$$(\eta, x_i)>0, \quad \forall i\in \N.$$
Also, because $(\eta, \eta')=0$ and $\eta'\in \overline{Z}\subset U^*$, it follows from Lemma~\ref{orth} that
\begin{equation}\label{infmany}
  (\eta', x_i)\geq 0, \quad \forall i\in \N.
\end{equation} 
Now suppose that
$$\limsup_{i\to\infty}(\eta', x_i)|x_i|_1 >0.$$ 
Then Corollary~\ref{convexCom} yields that 
$$\limsup_{i\to\infty}(\eta, x_i)|x_i|_1 >0,$$ 
contradicting the assumption
$$\lim_{i\to\infty}(\eta, x_i)|x_i|_1=0.$$
Hence 
$$\limsup_{i\to\infty}(\eta', x_i)|x_i|_1 \leq 0, $$
and in view of (\ref{infmany}) we have 
$$\lim_{i\to\infty}(\eta', x_i)|x_i|_1 = 0.$$
\end{proof}

\begin{proposition}
\label{limsupEqual}
Suppose that $(x_i)_{i\in \N}\subset \Phi^+$ and $(y_i)_{i\in \N}\subseteq \Phi^+$ are two infinite injective sequences of positive roots totally ordered by dominance. Moreover, suppose that 
$$x_{i+1}\dom y_{i+1}\dom x_i\dom y_i, \quad \forall i\in \N,$$
and 
$$\lim_{i\to\infty} \widehat{x_i}=\eta_1 \text{ and }\lim_{i\to\infty}\widehat{y_i}=\eta_2.$$ 
Furthermore, suppose that 
$$\limsup_{i\to\infty}(\eta_2, x_i)|x_i|_1 >0 \text{ and } \limsup_{i\to\infty}(\eta_1, y_i)|y_i|_1 >0.$$
Then
$$\eta_1=\eta_2.$$
\end{proposition}
\begin{proof}
Since $$\limsup_{i\to\infty}(\eta_2, x_i)|x_i|_1 >0,$$
it follows from Corollary~\ref{convexCom} that 
$$\limsup_{i\to\infty}(\eta_1, x_i)|x_i|_1 >0,$$
and $\eta_1\in\Ext$.

On the other hand, since $$\limsup_{i\to\infty}(\eta_1, y_i)|y_i|_1 >0,$$ and $\eta_1\in\Ext$, it follows from Proposition~\ref{limsup} 
that $\eta_2=\eta_1$.
\end{proof}


\bibliographystyle{amsplain}

\providecommand{\bysame}{\leavevmode\hbox to3em{\hrulefill}\thinspace}
\providecommand{\MR}{\relax\ifhmode\unskip\space\fi MR }
\providecommand{\MRhref}[2]{%
  \href{http://www.ams.org/mathscinet-getitem?mr=#1}{#2}
}
\providecommand{\href}[2]{#2}

\end{document}